\documentclass{article}
\usepackage{graphicx} 
\usepackage{amssymb,amsthm}
\usepackage{amsmath,amscd}

\usepackage{enumitem,kantlipsum}
\usepackage{cite}
\usepackage{xcolor, soul}
\usepackage{hyperref}
\usepackage{mfl}
\usepackage{orcidlink}

\title{Superabelian logics}
\author{Petr Cintula\footnote{Institute of Computer Science, Czech Academy of Sciences,
Pod Vodárenskou věží 271/2
182 00 Prague 8
Czech Republic.} \orcidlink{0000-0002-3617-1392}
\and
Filip Jankovec\footnotemark[1]
\footnote{Department of Algebra, Faculty of Mathematics and Physics, Charles University,
Sokolovská 49/83, 
186 75 Prague 8
Czech Republic.} \orcidlink{0009-0002-2746-0042}
\and
Carles Noguera \footnote{Department of Information Engineering and Mathematics,  University of Siena, 
         San Niccol\`o, via Roma 56, 53100 Siena, Italy.} \orcidlink{0000-0003-4910-599X}}

\newtheorem{thm}{Theorem}[section]
\newtheorem{lemma}[thm]{Lemma}
\newtheorem{prop}[thm]{Proposition}
\newtheorem*{obs}{Observation}
\newtheorem{con}[thm]{Conjecture}
\newtheorem{defn}[thm]{Definition}
\newtheorem{cor}[thm]{Corollary}
\newtheorem{example}[thm]{Example}

\newenvironment{proofmiddle}{
  
  \begin{proof}[Proof]
}{\end{proof}}

\newcommand{\oa}{\alg{1}} 
\newcommand{\A}{\alg{A}}  
\newcommand{\B}{\alg{B}} 

\newcommand{\pAb}{{\ensuremath{\logic{{pAb}}}}}
\newcommand{\rAb}{{\ensuremath{\logic{{rAb}}}}}
\newcommand{\Ab}{{\ensuremath{\logic{{Ab}}}}}
\newcommand{\Lu}{\ensuremath{\logic{{Lu}}}}

\newcommand{\tr}{\ensuremath{\mathtt{t}}}
\newcommand{\fa}{\ensuremath{\mathtt{f}}}

\renewcommand{\Alg}{\ensuremath{\mathbf{Alg}}}

\renewcommand{\R}{\ensuremath{\alg{R}}}
\renewcommand{\Q}{\ensuremath{\alg{Q}}}
\renewcommand{\Z}{\ensuremath{\alg{Z}}}

\newcommand{\f}{\ensuremath{\varphi}}
\newcommand{\p}{\ensuremath{\psi}}
\newcommand{\x}{\ensuremath{\chi}}

\date{ }

\begin{document}

\maketitle

\begin{abstract}
This paper presents a unified algebraic study of a family of logics related to Abelian logic ($\Ab$), the logic of Abelian lattice-ordered groups. We treat $\Ab$ as the base system and refer to its expansions as \emph {superabelian logics}. The paper focuses on two main families of expansions. First, we investigate the rich landscape of infinitary extensions of $\Ab$, providing an axiomatization for the infinitary logic of real numbers and showing that there exist $2^{2^\omega}$ distinct logics in this family. Second, we introduce \emph{pointed Abelian logic} ($\pAb$), the logic of pointed Abelian lattice-ordered groups, by adding a new constant to the language. This framework includes \emph{\L ukasiewicz unbound logic}. We provide axiomatizations for its finitary and infinitary versions as extensions of $\pAb$ and establish their precise relationship with standard Łukasiewicz logic via a formal translation. Finally, the methods developed for this analysis are generalized to axiomatize the logics of other prominent pointed groups.
\end{abstract}


\section{Introduction}

Robert Meyer and John Slaney, at a meeting of the Australasian Association for Logic of 1979,\footnote{Their work was only published ten years later in~\cite{Meyer-Slaney:AbelianLogic}, the same year in which Ettore Casari also published an independent discovery of the same logic motivated by the study of comparative adjectives in natural language~\cite{Casari:ComparativeLogics}.} presented a new logic system, {\em Abelian logic}, as a particular relevant logic in a language consisting of binary connectives $\imp$, $\conj$, $\lor$, and $\land$ and the truth constant $\tr$.\footnote{Note that the original paper uses two truth constants $\tr$ and $\fa$ which, however, were assumed to coincide. Later we give convincing reasons for choosing otherwise here.} Special features of Abelian logic are (1) the {\em axiom of relativity} $((\f\to\p) \to \p)\to \f$, a generalization of the double negation elimination law in which any formula $\p$ can take the place of falsity (2) its equivalent algebraic semantics, the variety of Abelian lattice-ordered groups (Abelian $\ell$-groups for short; see the next section for details). 

While one needs to work with all Abelian $\ell$-groups to get a sound and complete semantics for the full consequence relation of Abelian logic, if we only want to speak about its tautologies (or even the consequence relation restricted to finite sets of premises), it suffices to restrict to any given non-trivial Abelian $\ell$-group, in particular the usual additive group of the real numbers.

The tradition of endowing logical systems with an algebraic semantics based on algebras over real numbers is much older. Indeed, {\L}ukasiewicz logic (which can be defined in the language of Abelian logic expanded by a truth constant $\fa$ for falsum) in its infinitely-valued version was introduced in 1930 by {\L}ukasiewicz and Tarski~\cite{Lukasiewicz-Tarski:Untersuchungen} and since then it has proved to be one of the most prominent members of the family of many-valued logics often used to model some aspects of vagueness. Also, it has deep connections with other areas of mathematics such as continuous model theory, error-correcting codes, geometry, algebraic probability theory, etc.~\cite{Cignoli-Ottaviano-Mundici:AlgebraicFoundations,Gabbay-Metcalfe:ContinuousUninorms,DiNola-Leustean:Handbook, Mundici-LogicUlamGame}. 

Let us now introduce the real-valued semantics 
of these two logics side-by-side, to see their close relationship.

\begin{center}
\begin{tabular}{l@{\qquad}l@{\qquad}l}
\mbox{} & Abelian logic & \L ukasiewicz logic \\\hline
truth-values & $\mathbb{R}$ & $[0,1]$ \\
truth-definition  & $e(\f) \geq 0$ & $e(\f) = 1$  \\
$e(\f\to\p)$   & $e(\p) - e(\f)$ & $\min\{1, 1+ e(\p)- e(\f) \}$  \\
$e(\f\conj\p)$   & $e(\f) + e(\p)$ & $\max\{0, e(\f) + e(\p) -1 \}$   \\
$e(\f\wedge\p)$   & $\min\{e(\f) , e(\p) \}$ & $\min\{e(\f) , e(\p) \}$   \\
$e(\f\vee\p)$   & $\max\{e(\f) , e(\p) \}$ & $\max\{e(\f) , e(\p) \}$   \\
$e(\tr)$  & $0$ & $1$  \\
$e(\fa)$   & \mbox{-} & $0$    
\end{tabular}
\end{center}

We say that a formula $\f$ is a tautology of any of these two logics, if the truth definition holds for all evaluations (we get to consequence relations later). The tautologies of \L ukasiewicz logic form a proper subset of the classical tautologies (for example, the evaluation $e(p) = 0.5$ shows that the classical tautology $p \lor (p \to \fa)$ is not a tautology of \L ukasiewicz logic). In contrast, Abelian logic is contraclassical because, for instance, the formula $((p\to q) \to q)\to p$ is a tautology (while in classical logic we can take the evaluation $e$ such that $e(p)=0$ and $e(q)=1$, which makes it false). Observe also that Abelian logic is consistent: for instance, the evaluation $e(p) = e(q) = 1$ shows that $q\to(p \to q)$ is not a tautology.

Recently, a variation of \L ukasiewicz logic, that we will call here \emph{\L ukasiewicz unbound logic}, has been introduced in~\cite{Cintula-Grimau-Noguera-Smith:DegreesTo11}, with philosophical and linguistic motivations related to a finer analysis of reasoning with vagueness and graded predicates. This logic shares the language with \L ukasiewicz logic and so we can again present their semantics side-by-side for comparison (note that the semantics of the unbound logic uses all real numbers as truth degrees and hence it allows us to avoid  ``truncations'' in the definition of evaluations):
\begin{center}
\begin{tabular}{l@{\qquad}l@{\qquad}l}
\mbox{} & \L ukasiewicz unbound logic & \L ukasiewicz logic \\
truth-values & $\mathbb{R}$ & $[0,1]$ \\
truth-definition  & $e(\f) \geq 1$ & $e(\f) = 1$  \\
$e(\f\to\p)$   & $1+ e(\p) - e(\f)$ & $\min\{1, 1+ e(\p)- e(\f) \}$  \\
$e(\f\conj\p)$   & $e(\f) + e(\p) - 1 $ & $\max\{0, e(\f) + e(\p) -1 \}$   \\
$e(\f\wedge\p)$   & $\min\{e(\f) , e(\p) \}$ & $\min\{e(\f) , e(\p) \}$   \\
$e(\f\vee\p)$   & $\max\{e(\f) , e(\p) \}$ & $\max\{e(\f) , e(\p) \}$   \\
$e(\tr)$  & $1$ & $1$  \\
$e(\fa)$   & $0$ & $0$    
\end{tabular}
\end{center}

The first starting point of this paper is the observation that \L ukasiewicz unbound logic can be seen as an expansion of Abelian logic (in the language expanded with $\fa$): indeed we only need to use the isomorphism $f(x) = x - 1$ to obtain a ``shifted'' version of the semantics of \L ukasiewicz unbound logic:
\begin{center}
\begin{tabular}{l@{\qquad}l@{\qquad}l}
\mbox{} & Abelian logic & \L ukasiewicz unbound logic, shifted \\
truth-values & $\mathbb{R}$ & $\mathbb{R}$ \\
truth-definition  & $e(\f) \geq 0$ & $e(\f) \geq 0$  \\
$e(\f\to\p)$   & $e(\p) - e(\f)$ & $e(\p)- e(\f)$  \\
$e(\f\conj\p)$   & $e(\f) + e(\p)$ & $ e(\f) + e(\p) $   \\
$e(\f\wedge\p)$   & $\min\{e(\f) , e(\p) \}$ & $\min\{e(\f) , e(\p) \}$   \\
$e(\f\vee\p)$   & $\max\{e(\f) , e(\p) \}$ & $\max\{e(\f) , e(\p) \}$   \\
$e(\tr)$  & $0$ & $0$  \\
$e(\fa)$   & \mbox{-} & $-1$    
\end{tabular}
\end{center}
It is trivial to see that tautologies given by the original and the ``shifted'' semantics coincide and thus any tautology of Abelian logic is also a tautology of \L ukasiewicz unbound logic (its relation to \L ukasiewicz logic is more complex and is explored in Theorem~\ref{t:tau}).\footnote{\label{AbelLang}Note that this is the reason why we do not allow $\fa$ to be present in the language of Abelian logic and coincide with $\tr$ as this would make $\tr\to\fa$ one of its tautologies which is clearly not a tautology of \L ukasiewicz unbound logic.} This observation induces a rich family of logics \emph{expanding} Abelian logic in the language with just one additional truth constant.

The second starting point of the paper is the observation that, despite the fact that Abelian logic has no {\em finitary} consistent extensions in the original language, it has a vast family of \emph{infinitary} extensions.

Thus, the main goal of this article is twofold: (1) clarify the exact relations between Abelian, \L ukasiewicz and \L ukasiewicz unbound logics and (2) study two prominent families of expansions of Abelian logic: infinitary extensions in the original language and expansions in the language with $\fa$.

The paper is organized as follows. After this introduction, in Section~\ref{s:Prelim} we formally introduce Abelian logic \Ab\ and the theoretical framework to study its expansions. We choose the framework of (abstract) algebraic logic (see e.g.~\cite{Font-Jansana-Pigozzi:SurveyAAL,Font:AALBook}) to develop our study. More precisely, since $\Ab$ is a weakly implicative logic in the sense of~\cite{Cintula:WIFL-BasicProperties}, we capitalize mostly on the kind of abstract algebraic logic presented in the monograph~\cite{Cintula-Noguera:TheBook},  narrowing down its results and reformulating them in a convenient way for our purposes.
We also need to employ some tools of universal algebra which can be found e.g.\ in the monographs~\cite{Birkhoff:LatticeTheory,Burris-Sankappanavar:CourseUniversalAlgebra,Galatos-JKO:ResiduatedLattices,Gobrunov:Quasivarieties,Maltsev:Metamathematics}. In particular, in this approach, we have that axiomatic extensions correspond to subvarieties of algebras, finitary extensions correspond to subquasivarieties, and infinitary extensions to generalized subquasivarieties. Since, as mentioned above, every non-trivial Abelian $\ell$-group generates the whole class as a quasivariety, it turns out that there are no non-trivial subquasivarieties, and hence the only finitary extension of \Ab\ is the inconsistent logic (i.e.\ the logic in which all formulas are tautologies). While this section is thus mostly preliminary, it also presents a new, purely syntactic proof of the semilinearity of Abelian logic in Theorem~\ref{t:semilinearity of Ab}, which, to the best of our knowledge, has not appeared in this form before.

In contrast, the study of its infinitary extensions is much richer, for which in Section~\ref{s:InfAbel} we just scratch the surface. In particular, we give an infinitary rule that axiomatizes the extension of \Ab\ corresponding to the generalized quasivariety generated by \R. 
 Moreover, we provide $2^{\omega}$ semilinear infinitary extensions of $\Ab$, with corresponding completeness theorems, showing ${\vDash_{\R}} \subsetneq {\vDash_{\Q}}$ and that there are $2^{2^\omega}$ infinitary extensions of \Ab. At the end of the section, we discuss a possible axiomatization of the logic generated by \Z.
 
After that, in Section~\ref{s:PointedAbel}, we introduce the pointed Abelian logic \pAb, which gives rise to a dramatically different landscape, since here sub(quasi)varieties of algebras are plentiful. In particular, we prove that \pAb\ is finitely strongly complete with respect to $\{\R_{-1},\R_0,\R_1\}$. 

In Section~\ref{s:LukUnbound} we study the \L ukasiewicz unbound logics \Lu\ and $\Lu_\infty$, we axiomatize them as extensions of \pAb, and we show there is a translation of \L ukasiewicz logic into these logics. Following this, we generalize the methods to axiomatize and provide completeness results for a family of other prominent extensions of \pAb, which are similar to \Lu. We conclude the paper with Section~\ref{s:Conclusions} discussing possible avenues of further research.

\section{Abelian logic and its expansions}\label{s:Prelim}

Let $\Fm{L}$ be the set of all formulas built from a propositional language $\lang{L}$ and a countable set of variables, and let $\alg{\Fm{L}}$ be the absolutely free algebra of type $\lang{L}$ defined on $\Fm{L}$.

\begin{defn}
We say that a relation $\logic{L}$ between sets of formulas and formulas in a language $\lang{L}$ is a {\em logic} in ${\lang{L}}$ when it satisfies the following conditions for each $\Gamma\cup\Delta\cup\{\f\}\subseteq\Fm{L}$:
\begin{flushleft}
\begin{itemize}
\item $\{\f\}\vdash_\logic{L}\f$. \hfill (Reflexivity)\index{reflexivity}
\item If\/ $\Gamma\vdash_\logic{L}\f$ and $\Gamma \subseteq \Delta$, then $\Delta\vdash_\logic{L}\f$. \hfill (Monotonicity)
\item If\/ $\Delta\vdash_\logic{L}\p$ for each $\p\in \Gamma$ and\/ $\Gamma\vdash_\logic{L}\f$, then $\Delta\vdash_\logic{L}\f$. \hfill (Cut)\index{cut}
\item If\/ $\Gamma\vdash_\logic{L}\f$, then $\sigma[\Gamma]\vdash_\logic{L}\sigma(\f)$ for each substitution $\sigma$, i.e.\ endomorphisms on $\alg{\Fm{L}}$. \hfill (Structurality)
\end{itemize}
\end{flushleft}
A logic $\logic{L}$ is said to be \emph{finitary} if furthermore
\begin{flushleft}
\begin{itemize}
\item If\/ $\Gamma\vdash_\logic{L}\f$, then there is a \emph{finite} $\Gamma'\subseteq\Gamma$ such that $\Gamma'\vdash_\logic{L}\f$. \hfill (Finitarity)
\end{itemize}
\end{flushleft}
As a matter of convention we say that $\logic{L}$ is {\em infinitary} if it is not finitary.
\end{defn}

A \emph{consecution} in a language $\lang{L}$ is a tuple written as $\Gamma \wdash \f$, where $\Gamma \cup \{\f\} \subseteq \Fm{L}$; it is said to be \emph{finitary} if $\Gamma$ is finite and \emph{infinitary} otherwise, we also identify consecutions $\emptyset \wdash \f$ with just the formula $\f$. Note that any logic can be seen as a set of consecutions and thus whenever $\Gamma \vdash_\logic{L} \f$, we say that the consecution $\Gamma\wdash\f$ is \emph{valid} in $\logic{L}$ or that $\logic{L}$  \emph{satisfies} $\Gamma\wdash\f$.

A set of consecutions closed under arbitrary substitutions (i.e.\ endomorphisms on $\alg{\Fm{L}}$, which are extended to consecutions in an obvious way) is called an \emph{axiomatic system}. An element $\Gamma \wdash \f$ of the set is called an \emph{axiom}, a \emph{finitary rule} or an \emph{infinitary rule} depending on the cardinality of $\Gamma$. An axiomatic system is called \emph{infinitary} if it contains at least one infinitary rule, otherwise it is called \emph{finitary}.

Given an axiomatic system $\AS$ we write $\Gamma\vdash_\AS \f$ if there is a proof of $\f$ from $\Gamma$ in $\AS$. Note that, as we allow infinitary rules, our notion of proof cannot be the usual finite sequence of formulas, but it has to be a tree with no infinite branches (for details see~\cite{Cintula-Noguera:TheBook}). We say that $\AS$ is a presentation (an axiomatic system) of logic $\logic{L}$ if $\logic{L} = {\vdash_\AS}$. Consequently, we can trivially observe that every (finitary) logic possesses a (finitary) presentation.

In this paper we deal with several logics in different languages with rather intricate interrelations. Therefore, we need to introduce the formal notions of extension and expansion.

\begin{defn}\label{d:expansions}
Let $\lang{L_1} \subseteq \lang{L_2}$ be two languages and\/ $\logic{L_i}$ a logic in $\lang{L_i}$ for $i \in \{1,2\}$. We say that $\logic{L_2}$ is
\begin{itemize}\addtolength\itemsep{2pt}
\item the {\em expansion of\/ $\logic{L_1}$ by a set of consecutions $\Con{S}$} in $\lang{L}_2$ if it is the least logic in $\lang{L_2}$ containing $\logic{L_1}$ and $\Con{S}$ (i.e.\
$\logic{L_2}$ is axiomatized by all substitutional instances in $\lang{L_2}$ of consecutions from $\Con{S}\cup\AS$, for any presentation $\AS$ of\/ $\logic{L}_1$). We denote $\logic{L_2}$ by $\logic{L_1} + \Con{S}$.

\item an (axiomatic) {\em expansion of\/ $\logic{L_1}$} if it is the expansion of\/ $\logic{L_1}$ by some set of {\em consecutions} (resp.\ axioms).

\item an {\em extension of\/ $\logic{L_1}$} if\/ $\logic{L_2}$ is an expansion of\/ $\logic{L_1}$ and $\lang{L_1}=\lang{L_2}$.
\end{itemize}
\end{defn}

Now we can formally define the \emph{Abelian logic} $\Ab$ as the logic in the language\footnote{The literature has equivalent presentations of the logic in other languages, with interdefinable sets of connectives. For example, one can have $\imp$ as a primitive connective and use it to define: the constant $\tr$ (as $\tr:=p \imp p$), a negation $\neg$ (as $\neg \f:= \f \imp \tr$), and the conjunction $\conj$ (as $\f \conj \p := \neg(\f \imp \neg \p) $); conversely, from $\neg$ and $\conj$, one can retrieve implication as $\f \imp \p := \neg(\f\conj \neg \p) $.} $\lang{L}_\Ab=  {\{\imp, \conj, \lor, \land, \tr\}}$ given by the  axiomatic system presented in Table~\ref{tab:Abel}.

\begin{table}
    \begin{formlist}
\rul{sf}	& $(\f \imp \p) \imp ((\p \imp \xi) \imp (\f \imp \xi))$ & suffixing\\
\rul{e}	& $(\f \imp (\p \imp \xi)) \imp (\p \imp(\f \imp \xi))$ & exchange\\
\rul{id}	& $\f \imp \f$ 			& identity\\

$\rul{rel}$	& $((\f \imp \p) \imp \p)  \imp \f$ 				& relativity axiom \\

\rul{ub_1}	& $\f \imp \f \lor \p$ 			& upper bound\\

\rul{ub_2}	& $\p \imp \f \lor \p$ 			& upper bound\\

\rul{lb_1} & $\f \land \p \imp \f$ 			& lower bound\\

\rul{lb_2} & $\f \land \p \imp \p$ & lower bound\\

\rul{sup} & $(\f \imp \x) \land (\p \imp \x) \imp (\f \lor \p \imp \x)$ & supremality\\

\rul{inf} & $(\x \imp \f) \land (\x \imp \p)  \imp (\x \imp \f \land \p) $ & infimality\\

\rul{push} & $\f \imp (\tr \imp \f) $ & push\\

\rul{pop} & $(\tr \imp \f)  \imp \f$ & pop\\

\rul{res_1}	& $(\f \imp (\p \imp \xi)) \imp (\f \conj \p \imp \xi)$ 			& residuation\\

\rul{res_2}	& $(\f \conj \p \imp \xi) \imp (\f \imp (\p \imp \xi)) $ 			& residuation\\

\MP	& $\f ,\f \imp \p \wdash \p$ & \emph{modus ponens} \\

\rul{Adj} & $\f ,\p \wdash \f \land \p$ & adjunction
\end{formlist}
\caption{The axiomatic system of Abelian logic}\label{tab:Abel}
\end{table}

\pagebreak

The following consecutions are known to be valid in $\Ab$ and thus they are valid in all its expansions:

\begin{formlist}
\rul{\tr} & $\tr$ & \\
\Tr & $\f \imp \p, \p \imp \xi \wdash \f \imp \xi$   & transitivity \\[2pt]
\rul{C_\lor} & $\f \lor \p \wdash \p \lor \f$ & $\lor$-commutativity\\[2pt]
\rul{A_\lor} & $(\f \lor \p) \lor \xi \wdash \f \lor (\p \lor \xi)$  & associativity\\[2pt]
\rul{I_\lor} & $\f \lor \f \wdash \f$ & $\lor$-idempotency\\[2pt]
\rul{p_\vee} & $(\f \imp \p) \lor (\p \imp \f)$ & prelinearity\\
    \rul{MP_\vee} & $\f \lor \p,\f \imp \p \wdash \p$  \\
\end{formlist}

Moreover, one can easily see that $\Ab$ is a weakly implicative logic in the sense of~\cite{Cintula-Noguera:TheBook}, i.e., it satisfies \rul{id}, \MP, \Tr\ and, for each $n$-ary $c$ in $\lang{L}_\Ab$, also  
\begin{formlistB}
\rul{sCNG_{\mathit c}} 	& $\f \imp \p, \p\imp \f \vdash_\Ab c(\xi_1,\dots,\xi_{i}, \f,\dots,\xi_n)\imp c(\xi_1,\dots,\xi_{i}, \p,\dots,\xi_n)$\\ 
		& \hfill for each $0 \leq i < n$.
\end{formlistB}

Expansions of $\Ab$ need not in general satisfy \rul{sCNG_{\mathit c}} for connectives added to the language $\lang{L}_\Ab$. Since we want to focus on weakly implicative logics, we provide the following definition.

\begin{defn}
A logic $\logic{L}$ in a language $\lang{L} \supseteq \lang{L}_\Ab$ is \emph{superabelian} if\/ $\logic{L}$ is an expansion of $\Ab$ satisfying \rul{sCNG_{\mathit c}} for each $c \in \lang{L} \setminus \lang{L}_\Ab$. 
\end{defn}

Notice the difference between the notion of superabelian logic and the well known terminology of {\em superintuitionistic} logics, which are {\em axiomatic extensions} (instead of {\em expansions}) of intuitionistic logic. Actually, as we shall soon see, $\Ab$ has no non-trivial axiomatic (or even finitary) extension.

There are many natural examples of superabelian logics. In Section~\ref{s:InfAbel} we will focus on infinitary extensions of $\Ab$ and in Sections \ref{s:PointedAbel} and \ref{s:LukUnbound} we focus on some extensions of pointed Abelian logic (the least expansion of Abelian logic in the language with the additional constant $\fa$). Other natural examples of superabelian logics are the expansions of $\Ab$ with connectives corresponding to additional arithmetical operations on real numbers, such as multiplication or division. Other examples would be various types of congruential modal Abelian logics, i.e., logics where:
$$
\f \imp \p,  \p \imp \f  \wdash \Box\f \imp \Box \p
$$

We introduce now the algebraic semantics of Abelian logic and all superabelian logics.

\begin{defn}
An \emph{Abelian lattice-ordered group} (\emph{Abelian $\ell$-group}, see e.g.~\cite{Fuchs:PartiallyOrdered}) is a structure of the form $\tuple{A, +, -,\lor,\land,0}$ such that:
\begin{itemize}
\item $\tuple{A, +, -, 0}$ is an Abelian group, i.e.\ $+$ is a commutative and associative binary operation with unit $0$ and $-$ is the inverse operation (that is, for each $a\in A$, we have $a + (-a) = 0$),
\item $\tuple{A,\lor,\land}$ is a lattice, i.e.\ the binary relation $\leq$ defined as $a \leq b$ iff $a \lor b = b$ iff $a \land b = a$ turns out to be an order in which, for each $x,y \in A$, $x \lor y$ and $x\land y$ are respectively the supremum and the infimum of $\{x,y\}$,
\item it satisfies the monotonicity condition, that is, for each $a,b,c\in A$: $a \leq b$ implies $a+c \leq b+c$.
\end{itemize}
\end{defn}

It is well known that the defining conditions of Abelian $\ell$-groups can be expressed by means of equations, so they form a variety which we denote by $\mathbb{AL}$. Moreover, the lattice reduct of any Abelian $\ell$-group is distributive.

In order to match the language that we have chosen to introduce the logic \Ab, Abelian $\ell$-groups can be equivalently seen as structures $\tuple{\A,+,\imp,\lor,\land,0}$ based on the interdefinability: $-a:=a \imp 0$ and $a \imp b:=b-a$. To avoid excessive parentheses, we stipulate that the unary operation $-$ has precedence over all binary operations. Natural examples of (linearly ordered) Abelian $\ell$-groups are those of the integers $\Z$, the rational numbers $\Q$, and the real numbers $\R$. Note that $a\to b \geq 0$ iff $a\leq b$.

Let $\lang{L}$ be any propositional language containing $\lang{L}_\Ab$ and let $\A$ be an $\lang{L}$-algebra with an Abelian $\ell$-group reduct. An {\em $\A$-evaluation} is a homomorphism from $\alg{\Fm}_{\lang{L}}$ to $\A$. A consecution $\Gamma \wdash \f$ is valid in $\alg{A}$, $\Gamma \vDash_\A \f$ in symbols, if for every $\A$-evaluation $e$ it holds that, whenever $e(\p) \geq 0$ for every $\p \in \Gamma$, then also $e(\f)  \geq 0$. Given a class $\K$ of $\lang{L}$-algebras with an Abelian $\ell$-group reduct, we write $\Gamma \vDash_\K \f$ if $\Gamma \vDash_\B \f$ for each $\B \in \K$. It is well known that $\vDash_\K$ is a superabelian logic.

Let $\logic{L}$ be a superabelian logic in the language $\lang{L}$. By $\ralg{L}$ we denote the class of {\em $\logic{L}$-algebras,} i.e., $\lang{L}$-algebras $\alg{A}$ which have an Abelian $\ell$-group reduct and $\logic{L} \subseteq {\vDash_\alg{A}}$. For the purposes of this paper, we sometimes write that $\alg{A}$ is a \emph{model} of $\logic{L}$ instead of saying that $\alg{A}$ is an $\logic{L}$-algebra.

The usual techniques of algebraic logic allow us to prove that $\ralg{\Ab}=\mathbb{AL}$ and, for any superabelian logic $\logic{L}$,  we have the following completeness theorem: for each $\Gamma \cup \{\f\}$, we have 
$$
\Gamma \vdash_\logic{L} \f\qquad \text{iff}\qquad \Gamma \vDash_\ralg{L} \f.
$$ 

Furthermore, one can easily check that the following consecutions are valid in \Ab:
\begin{itemize}
\item[]$\f \wdash \tr \land \f \imp \tr$ 
\item[]$\f \wdash \tr \imp \tr \land \f$
\item[]$\{\tr \land \f \imp \tr,\tr \imp \tr \land \f\} \wdash \f $
\end{itemize}
Taken together, these consecutions give the condition \rul{Alg} from~\cite[Theorem~2.9.5]{Cintula-Noguera:TheBook}, which guarantees that \Ab\ (resp.\ any superabelian logic $\logic{L}$) is algebraically implicative (that is, weakly implicative and algebraizable in the sense of Blok and Pigozzi~\cite{Blok-Pigozzi:AlgebraizableLogics}) and the class $\mathbb{AL}$ (resp.\ $\ralg{L}$) is its equivalent algebraic semantics. Therefore,  there is a dual order isomorphism between the lattices of finitary extensions of $\Ab$ and subquasivarieties of $\mathbb{AL}$. 

We are interested, for a given superabelian logic $\logic{L}$,  in completeness theorems  with respect to various (classes of) $\logic{L}$-algebras, instead of the whole equivalent algebraic semantics. For reasons apparent later, we distinguish two kinds of completeness theorems.

\begin{defn}
A superabelian logic $\logic{L}$ enjoys the {\em (Finite) Strong $\K$-Completeness} with respect to a class $\K$ of $\logic{L}$-algebras, $\mathrm{(F)}$\SKC\ for short, if for every (finite) set $\Gamma\cup\{\f\}$ of formulas,
$$
\Gamma \vdash_\logic{L} \f\qquad \text{iff}\qquad \Gamma \vDash_\K \f.
$$
\end{defn}

A particularly interesting class of $\logic{L}$-algebras is that of those whose underlying lattice is linearly ordered; following the common notation in the literature (see~\cite{Cintula-Noguera:TheBook}), we denote this class as $\lalg{L}$ and introduce the following central notion. 

\begin{defn}
A superabelian logic $\logic{L}$ is \emph{semilinear} if it has the Strong $\lalg{L}$-Completeness.
\end{defn}

As we will see later in Example~\ref{e:semilinearity of superabelian logics}, not all superabelian logics have to be semilinear. There is a very useful characterization of semilinearity among superabelian logics with a countable presentation (which is, of course, the case for \Ab, all finitary superabelian logics and, as we will see later, even some prominent infinitary ones) using the notion of a strong disjunction (see~\cite[Definition 5.1.2]{Cintula-Noguera:TheBook} and note that the definition also requires the validity of $\f\wdash \f\vee \p$ and $\p\wdash \f\vee \p$, which is clearly the case in any superabelian logic).

\begin{defn}
The lattice connective $\lor$ is a \emph{strong disjunction} in a superabelian logic $\logic{L}$ if for each set  $\Gamma \cup\Phi\cup \Psi \cup \{\xi\}$ of formulas,
$$\frac{\Gamma,\Phi\vdash_\logic{L}\xi\qquad \Gamma,\Psi\vdash_\logic{L}\xi}{\Gamma,\Phi\lor \Psi\vdash_\logic{L}\xi},$$
 where $\Phi\lor \Psi$ is short for $\{\f \lor \p \mid \f \in \Phi, \p \in \Psi\}.$
\end{defn}

The next theorem follows directly from~\cite[Theorems 6.2.2 and 5.5.14]{Cintula-Noguera:TheBook}.

\begin{thm} \label{t:equivalence_disj_semilinearity}
Let \logic{L} be a superabelian logic with a countable presentation. Then, $\logic{L}$ is semilinear iff $\lor$ is a strong disjunction.
\end{thm}

The power of the theorem above is that there is an easy syntactical way (see Corollary~\ref{c:semilinearity_criterium}) to check whether $\vee$ is a strong disjunction using a simple syntactical construction.

\begin{defn}
Let $\Gamma \cup \{\varphi\} \subseteq  \Fm{L}$. The \emph{$\lor$-form} of\/ a consecution $\Gamma \wdash \f$ is the set of the consecutions $$(\Gamma \wdash \f) ^\lor=\{\Gamma \lor \p \wdash \f \lor \p \mid \p \in \Fm{L}\}.$$ 
\end{defn}

\begin{lemma}\label{l:about nabla-forms}

Let\/ $\logic{L}$ be a superabelian logic,  $\A$ an $\logic{L}$-algebra, and\/ $\Gamma\cup\{\f\}$ a set of formulas. If we have $\Gamma \lor \p \vDash_\A \f \lor \p$ for each $\p$, then we also have $\Gamma \vDash_\A \f$.
Moreover, if $\A \in \lalg{L}$, the converse implication also holds.
\end{lemma}

\begin{proof}
Take an evaluation $e$ such that, for each $\xi \in \Gamma$, we have $e(\xi) \geq 0$. Then clearly $e(\xi \lor \f) \geq 0$ ($\alg{A}$ has a lattice reduct). Thus by our assumption for $\p=\f$, we know that $e(\f) = e(\f \lor \f) \geq 0$.
	
For the converse implication, take an evaluation $e$ such that $e(\xi \lor \p) \geq 0$ for each $\xi \in \Gamma$. Since we assume that $\A$ is linearly ordered, we can derive that either $e(\p)  \geq 0$ or, for all $\xi \in \Gamma$, we have $e(\xi) \geq 0$.
In the latter case, the assumption $\Gamma \vDash_\A \f$ gives us $e(\f)  \geq 0$ and so in both cases we have $e(\f \lor \p)  \geq 0$.
\end{proof}

The following theorem is a particular version of \cite[Theorem 5.2.6]{Cintula-Noguera:TheBook} (clearly, the consecutions $\rul{I_\lor}$ and $\rul{C_\lor}$ mentioned in that theorem are valid in each superabelian logic).

\begin{thm} \label{t:strong disjunction}
Let\/ $\logic{L}$ be a superabelian logic with a presentation $\AS$.
Then, the following are equivalent: 
\begin{enumerate}
\item $\lor$ is a strong disjunction. 
\item $R^{\lor}\subseteq \logic{L}$ for each consecution $R\in \AS$.
\item $R^{\lor}\subseteq \logic{L}$ for each consecution $R\in \logic{L}$.
\end{enumerate}
\end{thm}

Now we are ready to give the promised direct alternative proof of semilinearity of $\Ab$, which originally was established using a completeness theorem (see e.g.\ \cite[Example 6.2.3]{Cintula-Noguera:TheBook}).

\begin{thm}\label{t:semilinearity of Ab}
$\Ab$ is a semilinear logic with strong disjunction $\vee$.
\end{thm}

\begin{proof}
As $\Ab$ has a countable presentation, it suffices to prove that $\vee$ is a strong disjunction which, thanks to~Theorem~\ref{t:strong disjunction}, follows if we prove that $\Adj^\lor$ and $\MP^\lor$ are valid in $\Ab$.

For the first one, let \A\ be an Abelian $\ell$-group and $e$ an $\A$-evaluation such that $e(\f \lor \xi) \geq 0$ and $e(\p \lor \xi) \geq 0$. Then, by the assumptions and distributivity of the lattice operations, we have:
$$
e((\f \land \p) \lor \xi) = e((\f \lor \xi) \land (\p \lor \xi)) = e(\f \lor \xi) \land e(\p \lor \xi)
\geq 0.
$$

To prove the validity of $\MP^\lor$, consider an evaluation $e$ such that $0 \leq e(\f \lor \xi)$ and  $0 \leq e((\f \imp \p) \lor \xi)$. This gives us $0=e(\f \lor \xi) \land 0=e((\f \imp \p) \lor \xi) \land 0$. Using the distributivity of the lattice operations, we obtain 
$$
0=e((\f \land \tr) \lor (\xi \land \tr)) = e(((\f \imp \p)  \land \tr) \lor (\xi \land \tr)).
$$
For simplicity, let us set $e(\f \land \tr) = a$, $e(\p) = b$, $e(\xi \land \tr) = c$ and $e((\f\imp \p)  \land \tr)=d$. As in $\ell$-groups it holds $(a \lor b) + c= (a+c) \lor (b+c)$, we obtain
$$
0=(a \lor c)+(d \lor c)=(a+d) \lor (a+c) \lor (c+d) \lor (2c). 
$$
Note that, due to monotonicity of addition, we have 
\begin{align*}
a+d & \leq e(\f)+ (e(\f) \imp e(\p) ) = b\\
(a+c) \lor (c+d) \lor (2c) & \leq  (0 + c)\vee (c+ 0) \vee (c+0)\leq c.
\end{align*}
Therefore, we obtain $0 \leq b \lor c=e(\p) \lor e(\xi \wedge \tr) \leq e(\p \lor \xi)$ as required.
\end{proof}

The following corollary is a direct consequence of Theorems~\ref{t:semilinearity of Ab}, 
\ref{t:equivalence_disj_semilinearity}, and~\ref{t:strong disjunction}.

\begin{cor}\label{c:semilinearity_criterium}
Let $\logic{L}$ be a superabelian logic expanding $\Ab$ by a set $\mathcal{S}$ of consecutions such that $R^\lor$ is valid in $\logic{L}$ for each $R \in \mathcal{S}$. Then, $\vee$ is a strong disjunction in~$\logic{L}$. If, furthermore, $\logic{L}$ has a countable presentation, then $\logic{L}$ is semilinear.
\end{cor}

In order to obtain completeness results for the logics studied in this paper, we will make use of the following two theorems and corollary which offer algebraic characterizations adapted from the general framework of the monograph~\cite{Cintula-Noguera:TheBook} to the setting of superabelian logics.
For clarity, we first recall the definitions of the standard class operators from universal algebra.
\begin{defn}
Let\/ $\K$ be a class of $\lang{L}$-algebras. We define the following:

\begin{itemize}
    \item $\mathbf{I}(\K)$, $\mathbf{S}(\K)$, $\mathbf{H}(\K)$, $\mathbf{P}(\K)$, $\PU(\K)$, and $\P_\omega(\K)$ are the classes of all isomorphic images, subalgebras, homomorphic images, direct products, ultraproducts, and countably-filtered products of algebras in $\K$, respectively.
    \item $\K_{\RFSI}$ is the class of relatively finitely subdirectly irreducible members of~$\K$.
    \item We use $\alg{1}$ to denote an arbitrary fixed trivial one-element algebra.
\end{itemize}
\end{defn}

\newpage

Since superabelian logics are algebraically implicative, we can recall the following two useful characterizations of (finite) strong completeness proved as \cite[Corollary 3.8.3 and Corollary 3.8.7]{Cintula-Noguera:TheBook} (in the latter case we also use the fact that for semilinear logics $\Alg^*(\logic{L})_\RFSI=\lalg{L}$, see~\cite[Theorem 6.1.7]{Cintula-Noguera:TheBook}).

\begin{thm}\label{t:semilinear completeness}
Let\/ $\logic{L}$ be a  superabelian logic and\/ $\K \subseteq \ralg{L}$. Then, the following are equivalent:
	
\begin{enumerate}
	\item $\logic{L}$ is strongly complete w.r.t.\ $\K$.
    \item $\ralg{L}$ is generated by $\K$ as a generalized quasivariety, i.e.\ $\ralg{L} = \ISP_\omega(\K)$.
\end{enumerate}
\end{thm}

\begin{thm}\label{t:semilinear finitary completeness}
Let $\logic{L}$ be a superabelian logic and $\K \subseteq \ralg{L}$. Then, the following are equivalent:
\begin{enumerate}
\item $\logic{L}$ is finitely strongly complete w.r.t.\ $\K$.
\item $\ralg{L}$ is generated by $\K$ as a quasivariety, i.e.\ $\ralg{L} = \mathbf{ISP}\PU(\K)$.
\end{enumerate}
If furthermore $\logic{L}$ is finitary and semilinear, then we can add
    \begin{enumerate}\setcounter{enumi}{2}
\item $\lalg{L} \subseteq \ISPU(\K,\oa)$.
\end{enumerate}
\end{thm}

Finally, we add a simple yet important corollary of the previous which follows from the fact that $\ISPU(\K_1,\K_2,\dots,\K_n)=\bigcup_{i \leq n} \ISPU(\K_i)$~\cite[Theorem 5.6]{Bergman:UniversalAlgebra}.

\begin{cor}\label{c:ultra-decomposition of semilinear logics}
Let\/ $\logic{L}$ be a finitary semilinear superabelian logic. Consider finitely many classes of\/ $\logic{L}$-algebras $\K_1,\K_2,\dots, \K_n \subseteq \ralg{L}$. Then, the following are equivalent:
		
\begin{enumerate}
\item $\logic{L}$ is finitely strongly complete w.r.t.\ $\bigcup_{i=1}^n \K_i$.
\item  $\lalg{L} \subseteq \I(\oa) \cup \bigcup_{i \leq n} \ISPU(\K_i)$.
\end{enumerate}
\end{cor}

Thanks to~\cite{Khisamiev:UniversalTheoryAbelianGroups}, we know the quasivariety of Abelian $\ell$-groups is generated (as a quasivariety) by $\Z$ and, since $\Z$ can be embedded into any non-trivial Abelian $\ell$-group, we obtain the following corollary.

\begin{cor} \label{c:Abel completeness}
The variety $\mathbb{AL}$ of Abelian $\ell$-groups is generated (as a quasivariety) by any of its  non-trivial members, that is, $\ISPPU(\alg{A})=\mathbb{AL}$ for any non-trivial $\alg{A} \in \mathbb{AL}$. Therefore, $\mathbb{AL}$ has no non-trivial subquasivarieties.
\end{cor}

As we will see in the next section, $\mathbb{AL}$ is not generated as a \emph{generalized} quasivariety by any of $\R$, $\Q$, or $\Z$, and it admits numerous non-trivial generalized subquasivarieties.

In logical terminology, we know that \Ab\ is finitely strongly complete w.r.t.\ any non-trivial Abelian $\ell$-group and has no consistent finitary extension, though it has numerous proper consistent infinitary extensions. Moreover, it is not strongly complete w.r.t.\ any of $\R$, $\Q$ or $\Z$. In the next section, we will axiomatize its proper extension strongly complete w.r.t.\ $\R$.

\section{Infinitary extensions of Abelian logic}\label{s:InfAbel}

Having established that Abelian logic ($\Ab$) lacks non-trivial finitary extensions, a well-known result that motivates our next step, we now turn our attention to the much richer landscape of its infinitary extensions. This section explores this domain by introducing infinitary rules designed to capture specific semantic properties.

Clearly, the most prominent candidates for such infinitary logics are ${\vDash_\R}$, ${\vDash_\Q}$, and ${\vDash_\Z}$. We know that:
$$
\Ab \subseteq {\vDash_\R} \subseteq {\vDash_\Q}  \subseteq {\vDash_\Z} 
$$
and that above ${\vDash_\Z}$ there is only the inconsistent logic (because $\Z$ obviously generates the smallest non-trivial generalized subquasivariety of Abelian $\ell$-groups). 

In this section, we obtain the following new results: 
\begin{itemize}
\item strictness of all three inclusions, and thus these three logics are indeed not finitary (Example~\ref{e:ArchNonArch}, Corollary~\ref{c:manyEXT} and Theorem~\ref{t:IDC})
\item there are $2^{2^\omega}$ distinct infinitary logics between ${\vDash_\R}$ and  ${\vDash_\Q}$ (Corollary~\ref{c:manyEXT}) 
\item an axiomatization of the logic $\vDash_\R$ (Theorem~\ref{t:realAbel}).
\end{itemize}
Moreover, we give a conjecture for an axiomatization of the logic $\vDash_\Z$, give strong arguments to support it, and point to the impossibility of proving it by usual methods (Conjecture~\ref{c:conjecture}).

From now on, we will use the following notation: for $n \in \N$ and a formula $\f$, we write $n \cdot \f$ instead of $\underbrace{\f \conj \cdots \conj \f}_{\text{$n$ times}}$ (we also set $0 \cdot \f = \tr$). We use the analogous notation for $+$ in the algebraic setting. 

Let us first recall the infinitary rule used to axiomatize the infinitary {\L}ukasiewicz logic (see~\cite{Hay:Axiomatization}):
\begin{equation*} \tag{Hay}
    \{\neg \f \imp n \cdot \f \mid n \in \N  \} \wdash \f
\end{equation*}
Since the language of $\Ab$ does not include a negation with similar properties to that of {\L}ukasiewicz logic, we have to replace $\neg \f$ by $\p$ and define:
\begin{equation*} \tag{Arch}
    \{\p \imp n \cdot \f \mid n \in \N  \} \wdash \f
\end{equation*}

As its name suggests, the rule \rul{Arch} is intended as a syntactic counterpart to the well-known algebraic property of being Archimedean. We introduce the standard definition of this property for Abelian $\ell$-groups.

\begin{defn}
Let $\A$ be an Abelian $\ell$-group. We say $\A$ is Archimedean if and only if for each $a,b \in A$ it holds:
$$(\forall n \in \mathbb N) \; na\leq b \implies  a \leq 0.$$
\end{defn}

The following lemma makes the connection between the rule and the pro\-perty precise, showing that the \rul{Arch} rule is valid in an Abelian $\ell$-group if and only if that group is Archimedean.

\begin{lemma}
Let $\A$ be an Abelian $\ell$-group. Then, the following conditions are equivalent:

\begin{enumerate}
\item $\A$ is Archimedean.
\item \rul{Arch} is valid in $\A$.
\end{enumerate}

Furthermore, if $\A$ is linearly ordered, we can add the following condition:
\begin{enumerate}\setcounter{enumi}{2}	
\item \rul{Arch}$^\lor$ is valid in $\A$.
\end{enumerate}
\end{lemma}

\begin{proof}
We first show the implication from 1 to 2.
Let $e$ be an evaluation such that we have $e(\f) \ngeq 0$. We set $a=-e(\f)$ and $b=-e(\p)$.
Since $\A$ is Archimedean and $a\nleq 0$, there is an $n \in \N$ such that $n \cdot a \nleq b$. Consequently, there is an $n \in \N$ such that
$n \cdot e(\f)  \ngeq e(\p) $; thus $e(\p \imp n \cdot \f) \ngeq 0$. This shows that \rul{Arch} holds.

To show the implication from 2 to 1, we assume that $\A$ is not Archimedean. Therefore, there exist $a,b \in A$ such that $a \nleq 0$ and $n \cdot a \leq b$ for each $n \in \N$. Consider an evaluation $e$ such that $e(\f)=-a$ and $e(\p)=-b$. Then, for each $n \in \N$, we have $n \cdot e(\f) \geq e(\p)$ and thus $e(\p \imp n \cdot \f) \geq 0$ for each $n \in \N$. However, by assumption, $e(\f)=-a \ngeq 0$ and thus \rul{Arch} fails in $\A$.

The rest of the statement follows from Lemma~\ref{l:about nabla-forms}.
\end{proof}

By the next example, we know that \rul{Arch} is not valid in $\Ab$ and so $\Ab$ is strictly contained in $\vDash_\R$ and thus $\vDash_\R$ cannot be finitary.

\begin{example}\label{e:ArchNonArch}
Clearly the algebras $\R$, $\Q$, and $\Z$ are Archimedean. 
As an example of a non-Archimedean linearly ordered $\ell$-group, consider the lexicographic product of two copies of the integers $\Z \times_{\text{lex}} \Z$, where $\Z \times_{\text{lex}} \Z$ has the same group structure as $\Z \times \Z$, with the ordering $\tuple{a,b} > \tuple{c,d}$ iff $a>c$ or $a=c$ and $b >d$.
This $\ell$-group is clearly not Archimedean, since by definition $n \cdot \tuple{0,1}<\tuple{1,0}$ for each $n \in \N$.
\end{example}

Let us call the logic $\vDash_\R$, the \emph{real Abelian logic} and denote it $\rAb$. The next theorem shows how to axiomatize it.

\begin{thm}\label{t:realAbel}
$\rAb$ is the extension of $\Ab$ by the rule $\rul{Arch}^\lor\!$.
\end{thm}

\begin{proof}
Let us denote the logic $\Ab+\rul{Arch}^\lor\!$ by $\logic{L}$. First note that, by Corollary~\ref{c:semilinearity_criterium}, we know that $\logic{L}$ is a semilinear logic and so it is strongly complete with respect to the class $\lalg{L}$, which, due to the previous lemma, consists of linearly ordered Archimedean $\ell$-groups. By Hölder's Theorem~\cite{Holder:Axiome}, each such $\ell$-group is embeddable into $\R$ and so the proof follows.
\end{proof}

Note that this proof relies on the $\lor$-form of the rule \rul{Arch} which is necessary to ensure that the resulting logic is semilinear. Therefore, it remains an open question whether the rule $\rul{Arch}$ itself would suffice on its own to axiomatize $\rAb$, which is equivalent to the question whether $\rul{Arch}^\lor$ is valid in $\Ab + \rul{Arch}$.

Next, we show that $\vDash_{\Q}$ is strictly stronger than $\vDash_\R$ (in other words, $\rAb$ is not strongly complete w.r.t.\ $\Q$).

We use the following convention. For an $\ell$-group $\A$, elements $a,b \in A$ and $\frac{k}{l} \in \mathbb Q \setminus \{0\}$ we write $\frac{k}{l} \cdot a \leq b$ or $a \leq \frac{l}{k} \cdot b$ when $k \cdot a \leq l \cdot b$.
Similarly, we write $\frac{k}{l} \cdot \f \imp \psi$ or $\f \imp \frac{l}{k} \cdot \psi$ instead of $k \cdot \f \imp l \cdot \psi$.

For each $0<\gamma \in \R \setminus \Q$, we introduce the following rule, in which $\tuple{q_j}_{j \in \N}$ is a strictly increasing sequence of rationals converging to $\gamma$ and $\tuple{r_j}_{j \in \N}$ is a strictly decreasing sequence of rationals converging to $\gamma$:
\begin{equation} \tag{str$_\gamma$}
    \{q_j \cdot \f \imp \p, \p \imp r_j \cdot \f \mid j \in \N\} \wdash  \f \imp \tr.
\end{equation}

The next proposition shows that in our context \rul{str_\gamma} does not depend on the choice of the sequences.

\begin{prop} \label{p:irrational}
Given $0<\gamma \in \R \setminus \Q$, a sequence of rationals $\tuple{q_j}_{j \in \N}$ strictly increasing and converging to $\gamma$, and a sequence of rationals $\tuple{r_j}_{j \in \N}$ strictly decreasing and converging to $\gamma$,
the logic $\rAb+\rul{str_\gamma}^\lor$ is strongly complete w.r.t.\ 
$$\K=\{\A \subseteq \R \mid (\forall a,b \in A) \;  (b=  \gamma \cdot a \implies a=0) \}.$$
Moreover, for any $\A \in \IS(\R)$ we have $\A \vDash \rul{str_\gamma}^\lor$ if and only if $\A \in \K$.
\end{prop}

\begin{proof}
    Since the logic $\rAb+\rul{str_\gamma}^\lor$ satisfies the conditions from Corollary~\ref{c:semilinearity_criterium}, it is semilinear. Since $\rAb+\rul{str_\gamma}^\lor$ is an extension of $\rAb$, we obtain that $\rAb+\rul{str_\gamma}^\lor$ is strongly complete w.r.t.\ some subclass of $\S(\R)$.

    We need the following technical observation:
    \begin{obs}
       Let $\A \subseteq \R$ and $a,b \in A$.
       Then we have:
       $$ q_j \cdot a \leq b \text{ and } b \leq r_j \cdot a \text{ for each } j \in \N \iff b=\gamma \cdot a \text{ and }a \geq 0.$$
    \end{obs}
    \begin{proofmiddle}
        First, since $q_0 \cdot a \leq b \leq r_0 \cdot a$ and $q_0 < r_0$, we can conclude that $0 \leq a$. 
        We have $\lim q_j=\gamma$ and $\lim r_j=\gamma$. Consequently, $\gamma \cdot a=\lim q_j \cdot a \leq b \leq \lim r_j \cdot a=\gamma \cdot a$. Therefore, we obtain $\gamma \cdot a=b$. The other direction follows as well: if $\gamma \cdot a=b$ and $a \geq 0$, then $q_j \cdot a \leq \gamma \cdot a=b$ for each $j \in \N$ (since $a \geq 0$ and $q_j$ is strictly increasing) and $b=\gamma \cdot a \leq r_j \cdot a$ for each $j \in \N$ (since $a \geq 0$ and $r_j$ is strictly decreasing).
    \end{proofmiddle}
  
First, let us take an arbitrary $\A \in \mathbb K$.
     We show that $\A$ satisfies $\rul{str_\gamma}^\lor$. 
     By the observation and the definition of $\K$, if, for $a,b \in A$, we have $q_j \cdot a \leq b$ and $b \leq r_j \cdot a$ for each $j \in \N$, then it must be that $a=0$. Thus, any evaluation in $\A$ that satisfies the premise of $\rul{str_\gamma}$ must map $\f$ to $0$, thereby satisfying the conclusion $\f \imp \tr$. Therefore, $\rul{str_\gamma}$ is valid in $\A$. Since $\A$ is a chain, by Lemma~\ref{l:about nabla-forms} $\rul{str_\gamma}^{\lor}$ is also valid in $\A$, and thus $\A$ is a model of $\rAb+\rul{str_\gamma}^\lor$.

    To prove the other direction, assume for an algebra $\oa \neq \A \subseteq \R$ that there exist $a,b \in A$ such that $b=\gamma \cdot a$ and $a \neq 0$.
    Clearly,  $\gamma \cdot (-a)=-b$ holds as well. Since $\A$ is totally ordered, we can assume without loss of generality that $a > 0$.
    By the observation, this implies that for each $j \in \N$, $q_j \cdot a \leq b$ and $b \leq r_j \cdot a$. But since $a,b \in A$ and $a > 0$, this is a counterexample to the validity of rule $\rul{str_\gamma}$ in $\A$. By Lemma~\ref{l:about nabla-forms}, $\rul{str_\gamma}^\lor$ fails in $\A$ as well. This shows $\alg A$ is not a model of $\rAb+\rul{str_\gamma}^\lor$.
\end{proof}

Clearly, for $\gamma_1,\gamma_2 \in \R \setminus \Q$, the rules $\rul{str_{\gamma_1}}^\lor$ and $\rul{str_{\gamma_2}}^\lor$ do not necessarily define different classes of models. For example, one can easily show that $\rul{str_{\sqrt{2}}}^\lor$ is valid in an algebra $\A$ if and only if $\rul{str_{2\sqrt{2}}}^\lor$ is valid in $\A$. Nevertheless, one can show the following.

\begin{cor}\label{c:manyEXT}
There are $2^{2^\omega}$ infinitary extensions of Abelian logic between $\vDash_\R$ and $\vDash_\Q$.
\end{cor}

\begin{proof}
It is a well-known result that the transcendence degree of $\R$ over $\Q$ is $2^\omega$. From this, it follows that there are $2^{2^\omega}$ distinct subfields of $\R$ that are field extensions of $\Q$. We argue that the $\ell$-group reduct of each such field is a model of a different infinitary extension of $\Ab$.
Let $\A$ and $\B$ be the $\ell$-group reducts of two distinct subfields of $\R$ that are field extensions of $\Q$. Without loss of generality, assume that there is a $\gamma \in A \setminus B$ such that $0 < \gamma$. By Proposition~\ref{p:irrational}, an $\ell$-subgroup of $\R$ validates the rule $\rul{str_\gamma}^\lor$ if and only if $b=\gamma \cdot a \implies a=0$ for all $a,b$ in the corresponding $\ell$-subgroup. Since $\alg A$ and $\alg B$ are the $\ell$-group reducts of subfields of $\R$, this condition holds if and only if $\gamma$ is not an element of their respective universes.
Hence we obtain that the rule $\rul{str_\gamma}^\lor$ is valid in $\B$ but is not valid in $\A$. This shows that $\vDash_\A$ and $\vDash_\B$ are different logics.
\end{proof}

Note that Proposition~\ref{p:irrational} also hints at a possible axiomatization of the logic $\vDash_\Q$. Indeed we can easily observe that $\vDash_\Q$ is axiomatized as $ \rAb+\{\rul{str_\gamma}^\lor \mid \gamma \in \mathbb R \setminus \mathbb Q\}$ iff this logic is semilinear (which we unfortunately do not know, as Corollary~\ref{c:semilinearity_criterium} works for countably axiomatized logics only). Therefore, the axiomatization of $\vDash_\Q$ remains as an open problem.

Finally, we show that $\vDash_\Z$ is strictly stronger than $\vDash_\Q$.  To this end, we propose the following infinitary rule IDC (standing for ``infinite decreasing chain''),
\begin{equation*} \tag{IDC}
\{ 2\cdot \f_{n+1} \imp \f_{n}, \f_n \imp 4 \cdot \f_{n+1} \mid n \in \N\} \wdash \f_0 \imp \tr.
\end{equation*}
and prove the following theorem:

\begin{thm}\label{t:IDC}
Let $\A$ be a non-trivial linearly ordered Archimedean Abelian $\ell$-group. Then \rul{IDC} is valid in $\A$ iff $\A$ is isomorphic to $\Z$. 
\end{thm}

Before we prove this theorem, let us state a conjecture based on this result.\footnote{Note that, by Example~\ref{ex:independence of Arch} we know that the rule $\rul{Arch}^\vee$ is not redundant.}

\begin{con}\label{c:conjecture}
The logic $\vDash_\Z$ is the extension of $\Ab$ by the rules $\rul{Arch}^\lor\!$ and $\rul{IDC}^\vee$.
\end{con}

Observe that the conjecture holds if and only if $\Ab+\rul{IDC}^\vee+\rul{Arch}^\vee$ is a semilinear logic. Unfortunately, the rule $\rul{IDC}$ uses infinitely many variables and thus, while we know that $\vee$ is a strong disjunction in $\Ab+\rul{IDC}^\vee+\rul{Arch}^\vee$ by Corollary~\ref{c:semilinearity_criterium}, we do not know whether this logic is actually a semilinear logic.

This is, however, a systematic problem as \emph{any} rule valid in $\vDash_\Z$ but not in $\vDash_\Q$ \emph{has to} use infinitely many variables. Indeed, since any finitely generated subalgebra of $\Q$ is isomorphic to $\Z$, any rule with finitely many variables valid in $\vDash_\Z$ must also be valid in $\vDash_\Q$.

Let us now proceed towards the proof of Theorem~\ref{t:IDC}. First, we need to introduce a machinery of decreasing sequences (which justifies the name of the rule).

\begin{defn}
Let $\A$ be an Abelian $\ell$-group. A sequence ${\tuple{a_i}}_{i \in \N}$ of elements of $A$ is \emph{strongly decreasing} if, for each $i \in \N$, there is an $n_{i} \geq 4$ such that:
$$
n_{i} \cdot a_{i+1} \geq  a_i \geq 2 \cdot a_{i+1}. 
$$
We call the sequence {\em trivial} if $a_i = 0$ for each $i$.
\end{defn}

Let us note that any non-trivial strongly decreasing sequence in $\alg{A}$ is indeed decreasing and $a_i \geq 0$ for each $i \geq 1$. The latter claim follows from the fact that for each $i$ there is an $n_i\geq 4$ such that $n_i \cdot a_i \geq 2n_i \cdot a_{i+1} \geq 2 \cdot a_i$, thus $(n_i-2) \cdot a_i \geq 0$, from which one can derive $a_i \geq 0$ (since in $\Z$ for any $n \geq 1$ the quasiequation $n \cdot x\geq 0 \implies  x\geq 0$ is valid, it holds in $\alg{A}$). The former claim then follows from the validity of $x \geq 0 \implies 2 \cdot x \geq x$ in $\Z$.

\begin{lemma} \label{l:strongly,strictly-decreasing}
Let ${\tuple{a_i}}_{i \in \N}$ be a non-trivial sequence of elements of an Abelian $\ell$-group $\A$. If ${\tuple{a_i}}_{i \in \N}$ is strongly decreasing, then it is strictly decreasing. 
If $\alg{A}$ is an Archimedean chain and it has a strictly decreasing sequence of positive elements, then it also has a non-trivial strongly decreasing sequence.
\end{lemma}

\begin{proof}
To prove the first claim, it suffices to show the strictness. Assume that $a_i=a_{i+1}$ for some $i \in \N$. Then the inequality $a_i \geq 2 \cdot a_{i+1}$ becomes $a_i \geq 2 \cdot a_i$ and thus $0 \geq a_i.$ Since also $a_i \geq 0$, we obtain $a_i=0$. For each $j > i$ we have $a_i \geq a_j \geq 0$, thus we get $a_j=0$ as well. Similarly, for $j < i$, there is an $n \in \N$ such that $n \cdot a_i \geq a_j \geq 0$ and thus again $a_j=0$. Therefore, the sequence ${\tuple{a_i}}_{i \in \N}$ has to be trivial, a contradiction.
    
We will prove the second claim. Let us assume that $\A$ is an Archimedean chain and ${\tuple{a_i}}_{i \in \N}$ is a strictly decreasing sequence of positive elements. Without loss of generality, as $\A$ is Archimedean, we can assume that $\A \subseteq \R$, by Hölder's Theorem.
We construct a non-trivial strongly decreasing sequence ${\tuple{b_i}}_{i \in \N}$. Set $b_0 > 0$ arbitrarily. We will define the rest of the sequence recursively. Assume that we have already constructed $\tuple{b_i}_{i \leq n}$ satisfying the conditions from the definition of a strongly decreasing sequence. Now we want a $b_{n+1}$ such that $0< 2 \cdot b_{n+1}\leq b_n$. 

Since ${\tuple{a_i}}_{i \in \N}$ is a strictly decreasing sequence, it follows that there exists a $j \in \N$ such that $\frac{a_j}{b_n} \notin \mathbb Z$ (otherwise $\tuple{\frac{a_j}{b_n}}_{j \in \N}$ would be a strictly decreasing sequence of positive elements in $\Z$, which is not possible). 

Since $\A$ is Archimedean, there is a $k \in \N$ such that $0 \leq a_j -k \cdot b_n < b_n$. Since $\frac{a_j}{b_n} \notin \mathbb Z$, we have $0 \neq a_j -k \cdot b_n$ and thus $0 < a_j -k \cdot b_n < b_n$.
Hence, $$0 < a_j -k \cdot b_n \leq \frac{b_n}{2} \; \text{ or } \; 0 < b_n-(a_j -k \cdot b_n) \leq \frac{b_n}{2}.$$ 
Therefore, we set$$b_{n+1}=\min\{a_j -k \cdot b_n, b_n-(a_j -k \cdot b_n)\}.$$ Clearly, $b_{n+1}>0$ and $2 \cdot b_{n+1}\leq b_n$. The second condition for strong decreasing sequences easily follows from the fact that $\A$ is Archimedean. Thus, ${\tuple{b_i}}_{i \in \N}$ is a non-trivial strongly decreasing sequence.
\end{proof}

\begin{example} \label{ex:independence of Arch}
We will show that there are non-Archimedean chains with strictly decreasing sequences of positive elements and with no non-trivial strongly decreasing sequence. Consider the $\ell$-group $\Z \times_{\text{lex}} \Z$ from Example~\ref{e:ArchNonArch}. 
First, let us consider the sequence ${\tuple{\tuple{1,-i}}}_{i \in \N}$. Clearly, this sequence is strictly decreasing and $\tuple{1,-i} >0$ for each $i \in \N$. We show that there is no non-trivial strongly decreasing sequence in $\Z \times_{\text{lex}} \Z$.

Let us assume that ${\tuple{\tuple{a_i,b_i}}}_{i \in \N}$ is a strongly decreasing sequence. Clearly $a_i \geq 0$ for all $i \in \N$ since $\tuple{a_i,b_i} \geq 0$. Thus, there has to be a $c \geq 0$ such that $a_i=c$ for infinitely many $i \in \N$. Therefore, there has to be a strongly decreasing subsequence ${\tuple{\tuple{c,d_i}}}_{i \in \N}$, where $d_i \in \mathbb Z$ and $c \geq 0$. Fix an arbitrary $i \in \N$. We know that $2 \cdot \tuple{c,d_{i+1}} \leq \tuple{c,d_{i}}$, thus $2\cdot c \leq c$. This implies $c=0$. 
    
Since ${\tuple{\tuple{0,d_i}}}_{i \in \N}$ is a strongly decreasing sequence in $\Z \times_{\text{lex}} \Z$, the sequence $\tuple{d_{i}}$ is strongly decreasing in $\Z$. Therefore, $d_i=0$ for all $i \in  \N$. This shows that ${\tuple{\tuple{c,d_i}}}_{i \in \N}$ has to be the trivial sequence and so is ${\tuple{\tuple{a_i,b_i}}}_{i \in \N}$ as well.
\end{example}

\begin{prop}
Let $\A$ be an Abelian $\ell$-group. Then \rul{IDC} is valid in $\A$ iff only the trivial sequence is strongly decreasing in $\A$.
\end{prop}

\begin{proof}

First, assume that $\A$ has a non-trivial strongly decreasing sequence ${\tuple{a_i}}_{i \in \N}$ and note that w.l.o.g.\ we can assume that $n_i = 4$ for each $i$ (if for some $i$ we have $a_i > 4\cdot a_{i+1}$, expand the sequence by putting $2\cdot a_{i+1}$ in a proper place, and repeat \dots). Let us consider the $\A$-evaluation $e$ induced by setting $e(x_i) = a_i$. Clearly, all the premises of \rul{IDC} are valid but we have $e(x_0 \imp \tr) = - a_0 < 0$ and thus the rule \rul{IDC} is not valid in $\A$.

Conversely, assume that the rule \rul{IDC} is not valid in $\A$, i.e., we have an evaluation $e$ such that if we set $e(\f_i) = a_i$ the validity of premises of \rul{IDC} tells us that $4 \cdot a_{i+1} \geq a_{i} \geq 2 \cdot a_{i+1}$, i.e., ${\tuple{a_i}}_{i \in \N}$ is a strongly decreasing sequence and, by the failure of the conclusion of the rule, we have $a_0 > 0$, i.e., the sequence is non-trivial.
\end{proof}

The previous proposition gives us a description of Abelian $\ell$-groups satisfying \rul{IDC} in terms of the existence of strongly decreasing sequences.
Lemma~\ref{l:strongly,strictly-decreasing} shows that there is a connection between strictly decreasing sequences of positive elements and strongly decreasing sequences. Therefore, it is natural to ask whether the previous proposition could be stated in terms of the existence of {\em strictly} decreasing sequences of positive elements. The following example shows that this is not the case.

\begin{example}
There exist $\ell$-groups in $\P(\Z)$ which have strictly decreasing sequences of positive elements but lack a non-trivial strongly decreasing sequence.

Consider the Abelian $\ell$-group $\Z^\omega$ and note that the sequence of positive elements $$
\tuple{\tuple{\underbrace{0,\dots,0}_{\text{i-times}},1,1,\dots}}_{i \in \N}
$$
is strictly decreasing. However, the previous proposition clearly states that only the trivial sequence is strongly decreasing in 
$\Z^{\omega}$, since $\Z^{\omega} \in \P(\Z)$, $\P(\Z) \subseteq \ISP_\omega(\Z)$, and \rul{IDC} holds in $\ISP_\omega(\Z)$.
\end{example}

Now we are finally ready to give the promised proof of Theorem~\ref{t:IDC}.

\begin{proof}[Proof of Theorem~\ref{t:IDC}]  

Let $\A$ be a non-trivial linearly ordered Archimedean Abelian $\ell$-group. We want to show that \rul{IDC} is valid in $\A$ if and only if $\A$ is isomorphic to $\Z$.
By the previous proposition, \rul{IDC} is valid in $\A$ if and only if $\A$ has only the trivial strongly decreasing sequence.

Since $\A$ is an Archimedean chain, by Lemma~\ref{l:strongly,strictly-decreasing}, having only the trivial strongly decreasing sequence is equivalent to having no strictly decreasing sequence of positive elements.

Thus, it suffices to show that a non-trivial linearly ordered Archimedean $\ell$-group $\A$ has no strictly decreasing sequence of positive elements if and only if $\A \cong \Z$.

Clearly, $\Z$ has no decreasing sequences of positive elements. To show the other implication, one has to argue that $\Z$ is the only non-trivial linearly ordered Archimedean $\ell$-group with no strictly decreasing sequences of positive elements. We can assume, by Hölder's Theorem, that without loss of generality, $\Z \subseteq \A \subseteq \R$ and $\A$ is not isomorphic to $\Z$. We have to distinguish two cases.
\begin{enumerate}

\item Let $\zeta \in\mathbb R\setminus \mathbb Q$ be an element of $A$. Without loss of generality we can assume $\zeta>1$. Since $1 \in A$, we can generate recursively the following sequence. We start with $r_0=\zeta$ and $r_1=1$. For $n> 0$, we define $k_n=\lfloor \frac{r_{n-1}}{r_n} \rfloor$ and let $r_{n+1}=r_{n-1}-k_n \cdot r_n$. 
By the properties of division, $0 \leq r_{n+1}<r_n$ for all 
$n \geq 1$. This means that the sequence $\tuple{r_n}_{n \in \N}$ is strictly decreasing. Furthermore, because $\zeta$ is irrational, the ratio $\frac{r_n}{r_{n+1}}$ is irrational for each $n \in \N$. This shows that this algorithm never terminates with a zero remainder; thus, $r_n>0$ for all $n$.

\item $\A \subseteq \mathbb \Q$. Since $\A$ is not isomorphic to $\Z$, $\A$ is not generated by any single element. Thus, for each $0<a_n \in A$, there is a $0<b_n \in A$ such that $a_n$ does not generate $b_n$. If $b_n<a_n$ we set $a_{n+1}=b_n$; otherwise, we set $k_n=\lfloor \frac{b_{n}}{a_n} \rfloor$ and 
$a_{n+1}=b_n- k_n \cdot a_n$. Thus, we have constructed $0<a_{n+1}<a_n$.
This allows us to build the desired sequence, and therefore we finish the proof.\qedhere
\end{enumerate}
\end{proof}

\section{Pointed Abelian logic}\label{s:PointedAbel}

In the previous sections, we have explored the landscape of Abelian logic and its infinitary extensions. We now turn our attention to expansions, that is, we allow not only new rules, but also enriched logical languages. 
This section introduces \emph{pointed Abelian logic} (\pAb), which extends Abelian logic with a new constant symbol \fa. 
This seemingly simple addition of an additional ``point'' to the algebraic semantics dramatically alters the character of the logic. Whereas Abelian logic has no non-trivial finitary extensions, we will show that this pointed version gives rise to a plentiful landscape of logical extensions.\footnote{All the semilinear extensions of $\pAb$ are characterized in the submitted article \cite{Jankovec:SubvarietiesPointedAbelianLGroups}.} The central result of this section is a completeness theorem demonstrating that this new logic is finitely strongly complete with respect to just three canonical pointed structures over the real numbers: $\R_{-1}, \R_0$, and $\R_1$.

Let us consider the expansion of the language $\lang{L_\Ab}$ of Abelian logic by a new constant symbol $\fa$. We denote this language by $\lang{L_\pAb}$. In this section, we study the minimal expansion of $\Ab$ in this language.

\begin{defn}
The pointed Abelian logic is the least expansion of $\Ab$ in the language $\lang{L_\pAb}$,
i.e., it is axiomatized by the same axiomatic system as $\Ab$ only closed under substitution of the expanded language.
\end{defn}

Recall that Full Lambek logic with exchange, $\FL_\ae$, is a prominent substructural logic in the language of $\lang{L_\pAb}$ which is axiomatized by omitting the axiom \rul{rel} from the axiomatic system given in Table~\ref{tab:Abel}. Thus, pointed Abelian logic is the extension of $\FL_\ae$ by axiom $\rul{rel}$. 

Interestingly, we can see $\Ab$ as an extension of $\pAb$ by the axiom which collapses both constants:
$$
(\fa\to\tr)\wedge(\tr\to\fa).
$$
Therefore, sometimes in the literature $\Ab$ is presented as the expansion of $\FL_\ae$ by axioms $\rul{rel}$ and $(\fa\to\tr)\wedge(\tr\to\fa)$. But, as we have commented in footnote \ref{AbelLang}, we prefer to see $\pAb$ as an expansion of $\Ab$. 

Clearly, $\pAb$ and all its extensions have \sCNG\ and hence they are superabelian logics. Therefore, it is easy to see that $\Alg^*(\pAb)$, i.e.\ the equivalent algebraic semantic of $\pAb$, is the class of \emph{pointed Abelian $\ell$-groups}, i.e.\ Abelian $\ell$-groups expanded to the language $\lang{L_\pAb}$ with a new constant symbol interpreted in an arbitrary way.

To be able to easily talk about particular $\pAb$-algebras, we present the following definition.

\begin{defn}
Let us have an algebra $\A$ in a language $\lang{L}$ such that $\fa \notin \lang{L}$ and take an element $a \in A$. By $\A_a$ we denote the algebra in the language $\lang{L} \cup \{\fa\}$, which is the expansion of $\A$ and where the constant symbol $\fa$ is interpreted as~$a$.
 
If\/ $\K$ is a class of algebras, by $\po\K$ we denote $\{\A_a \mid \A \in \K, a \in A\}$. For a singleton $\{\A\}$, we may write $\po \A$ instead of\/ $\po \{\A\}$.
\end{defn}

$\R_{-1}$, $\R_0$, and $\R_1$ are natural examples of pointed Abelian $\ell$-groups. 
Note that obviously, $\Alg^*(\pAb) = \po \mathbb{AL}$.
It is also a straightforward, albeit rather technical, matter to show that this pointing operator $\po$ commutes with the standard algebraic class operators for taking isomorphisms, subalgebras, products, and ultraproducts.

\begin{thm} \label{t:pAb lots R}
Let $\A$ be a non-trivial Abelian $\ell$-group.
$\pAb$ is finitely strongly complete w.r.t.\ $\po \A=\{\A_a\mid a \in \A\}$.
\end{thm}

\begin{proof}
Assume that $\Gamma\not\vdash_\pAb\f$. Given an $\lang{L}_\pAb$-formula $\p$, we define a $\lang{L}_\Ab$-formula $\p'$ by replacing all occurrences of $\fa$ by an arbitrary fixed variable $p$ not occurring in $\Gamma\cup\{\f\}$. Let us denote $\Gamma'=\{\p' \mid \p \in \Gamma\}$. Then, clearly, $\Gamma'\not\vdash_\Ab\f'$ and so $\Gamma'\not\vDash_\alg{A}\f'$ as witnessed by an evaluation $e$ which entails $\Gamma\not\vDash_{\alg{A}_{e(p)}}\f$.
\end{proof}

Let us formulate a trivial lemma which allows us to improve the previous result for $\A$ being either $\R$ or $\Q$ and will also be of use later.

\begin{lemma}\label{l:isomorRQ}
Let $\A$ be either $\R$ or $\Q$ and $a,b \in A\setminus\{0\}$. Then, $\A_a$ and $\A_b$ are isomorphic iff $a \cdot b > 0$.
\end{lemma}

\begin{proof}
If $\mathbf{A}_a$ and $\mathbf{A}_b$ are isomorphic, the isomorphism must preserve the underlying natural order, which trivially forces the designated units $a$ and $b$ to have the same sign; hence $a \cdot b > 0$. For the other direction, assume $a \cdot b > 0$. Then $\frac{b}{a} > 0$, and the mapping $h\colon x \mapsto \frac{b}{a} \cdot x$ is an order-preserving group automorphism mapping $a$ to $b$, providing the needed isomorphism.
\end{proof}

\begin{thm}\label{t:pAbcompl}
The logic $\pAb$ is finitely strongly complete w.r.t.\ $\{\R_{-1},\R_0,\R_1\}$ and w.r.t.\ $\{\Q_{-1},\Q_0,\Q_1\}$.
\end{thm}

Observe that the same would not work e.g.\ for the Abelian $\ell$-group of the integers, since $\Z_a$ and $\Z_b$ are not isomorphic unless $a=b$. In fact, one can show that $\Z_a$ generates a different variety for each $a \in \mathbb Z$.

\begin{prop}\label{p:varietyZ}
For any $a,b \in \mathbb{Z}$, we have $\HSP(\Z_a)=\HSP(\Z_b)$ if and only if $a=b$.
\end{prop}

\begin{proof}
    One direction is obvious. To prove the converse one, let us assume $a \neq b$. To show that $\HSP(\Z_a) \neq \HSP(\Z_b)$, it is sufficient to find an equation which is valid in one of the generating algebras, $\Z_a$ or $\Z_b$, but not in the other.
    If $a<0$ and $b \geq 0$ take the equation $\fa \geq 0$. This equation clearly holds in $\Z_b$ but it does not hold in $\Z_a$.
    The case $a \leq 0$ and $b>0$ is similar.
    Therefore, it remains to check the cases where $a,b>0$ or $a,b<0$.
    We will check only the first case (the second one is analogous).
    Moreover, let us assume without loss of generality that $0 <a<b$.
     
    Consider the following equation: 

    \begin{equation*} \tag{*}
        (a \cdot x -\fa) \lor (-x) \geq 0.
    \end{equation*}
    This equation is not valid in $\Z_b$, since if we take an evaluation $e$ such that $e(x)=1$ we get $(a \cdot e(x) -b) \lor (-e(x))=(a-b) \lor -1 < 0$. Thus the equation (*) does not hold in $\Z_b$.
    
    We show the same equation holds in $\Z_a$. First let us consider an evaluation $e$ such that $e(x) \leq 0$. Clearly, we get $-e(x) \geq 0$ and thus $(a \cdot e(x) - a) \lor e(-x) \geq 0.$  
    It remains to consider an evaluation $e$ such that $e(x)>0$ (and thus $e(x) \geq 1$). Then $a \cdot e(x) -a \geq a \cdot 1-a=0$. Thus the equation (*) is valid in $\Z_a$.    
\end{proof}

\section{Łukasiewicz unbound logic (and other logics of reals)}\label{s:LukUnbound}

In this section, we formally introduce Łukasiewicz unbound logic within our framework, motivated by the observation that its original algebraic semantics is isomorphic to the pointed Abelian $\ell$-group $\R_{-1}$. Our primary goal is to axiomatize this logic as an extension of pointed Abelian logic. We first define the finitary version, $\Lu$, and establish its finite strong completeness with respect to $\R_{-1}$. Then, we introduce its infinitary version, $\Lu_{\infty}$, by adding the Archimedean rule to achieve strong completeness, and make the connection to standard Łukasiewicz logic precise by providing a formal translation. Generalizing the methods developed for this specific case, we broaden our scope to systematically axiomatize the logics of other prominent pointed groups, such as $\R_0$ and $\R_1$. We establish a correspondence between algebraic properties of the point element $\fa$ and simple syntactic rules, allowing us to provide a comprehensive overview of completeness results for this family of logics.

Let us start this section by observing that the algebra $\alg{R}_{-1}$ is isomorphic to the algebra $\alg{LU}$ that was used in~\cite{Cintula-Grimau-Noguera-Smith:DegreesTo11} to introduce \L ukasiewicz unbound logic. $\alg{{LU}}$ is a member of $\po\mathbb{AL}$ defined over the reals with lattice operations defined in the usual way and other operations/constants defined as:
$$
x \to^\alg{LU} y = 1 - x +y \qquad x \conj^\alg{LU} y = x + y - 1 \qquad 
\tr^\alg{LU} = 1  \qquad \fa^\alg{LU} = 0
$$
Clearly, $f\colon \alg{R}_{-1} \longrightarrow \alg{LU}$ defined as $f(x) = x +1$ is the desired isomorphism. As already mentioned in the introduction, the operations on $\alg{LU}$ closely resemble those of the well-known standard MV-algebra $\mathbfL$ which provides semantics for the \L ukasiewicz logic (see more details at the end of this section). In the context of this paper we will see it as an algebra in the language of $\pAb$ (though it clearly is not a member of $\ralg{\pAb}$)\footnote{
Indeed, here we choose a nonstandard representation of \MV-algebras. It is easy to see that this representation is equivalent to the most common one (found e.g.\ in~\cite{Cignoli-Ottaviano-Mundici:AlgebraicFoundations}), the reader can check it as an exercise or check any basic literature related to \MV-algebras (again see e.g.~\cite{Cignoli-Ottaviano-Mundici:AlgebraicFoundations}).
} defined on the real unit interval $[0,1]$ with lattice operations defined in the usual way and other operations/constants defined as:
$$
x \to^\alg{\mathbfL} y = \min\{1,1 - x +y\} \qquad x \conj^\alg{\mathbfL} y = \max\{0,x + y - 1\} \qquad 
\tr^\alg{\mathbfL} = 1  \qquad \fa^\alg{\mathbfL} = 0
$$

Recall that we know that $\pAb$ is \emph{finitely} strongly complete with respect to $\{\alg{R}_{-1}, \alg{R}_{1}, \alg{R}_{0}\}$, thus our goal is to find a finitary rule which allows us to extend $\pAb$ in a way that preserves semilinearity and also ensures that the rule is not valid in $\R_0$ and $\R_1$.

\begin{defn}
\emph{\L ukasiewicz unbound logic} $\Lu$ is the extension of the logic $\pAb$ by the rule:
\begin{equation*}\tag{Lu}
   \fa \lor \f \wdash \f
\end{equation*}
\end{defn}

\begin{lemma}
$\Lu$ is a semilinear logic.
\end{lemma}

\begin{proof}
By Corollary~\ref{c:semilinearity_criterium}, it is enough to check that $(\fa \lor \f) \lor \p \wdash \f \lor \p$ is valid in $\Lu$. This clearly follows from the use of rules $\rul{A_\lor}$ and $\fa \lor \f \wdash \f$.
\end{proof}

To obtain the desired completeness, let us characterize linearly ordered Abelian $\ell$-groups satisfying the rule $\rul{Lu}$. We give a slightly more general result, which will allow us to show that we cannot replace $\rul{Lu}$ by a simpler, weaker rule that might initially seem sufficient.

\begin{lemma} \label{l:similar extensions}
Let $\A \in \ralg{\pAb}$ and $\A \neq \oa$. Then, the following conditions are equivalent:
\begin{enumerate}
\item $\fa^\A \not\geq 0$.
\item The rule $\fa \wdash \f$ is valid in $\A$.
\end{enumerate}

Furthermore, if $\A$ is linearly ordered, we can add the following condition:
\begin{enumerate}\setcounter{enumi}{2}	
\item The rule $\fa \lor \f \wdash \f$ is valid in $\A$.
\end{enumerate}
\end{lemma}

\begin{proof}
The implications from 1 to 2 and from 1 to 3 are clear (note that in the second case we need to use linearity of $\alg{A}$ to obtain that if $\fa^\A \not\geq 0$ and $\fa \vee \f \geq 0$, then $\f \geq 0$). We prove the converse ones at once by contradiction:
assume that $\fa \vDash_\alg{A} \varphi$ (resp.\ that  $\fa\vee \varphi \vDash_\alg{A} \varphi$)
and $\fa^\A \geq 0$. For an arbitrary element $a$ and any evaluation $e(\varphi) = a$, the premises of both rules are trivially valid and so we obtain $a \geq 0$, a contradiction with the assumption that $\A \neq \oa$.
\end{proof}

Thus, we know that in particular $\R_{-1} \in \lalg{\Lu}$ but $\R_0, \R_1 \not\in \lalg{\Lu}$. The next example shows that the rules  $\fa \lor \f \wdash \f$ and  $\fa \wdash \f$ are not interderivable and thus the extension of $\pAb$ by the rule $\fa \wdash \f$ is strictly weaker than $\Lu$.

\begin{example}\label{e:semilinearity of superabelian logics}
Let us consider the Abelian $\ell$-group $\R_{-1} \times \R_0$. Clearly, this Abelian $\ell$-group satisfies the rule $\fa \wdash \f$, but it does not satisfy $\fa \lor \f \wdash \f$.
\end{example}

\begin{thm}\label{t:Lu completeness}
$\Lu$ is finitely strongly complete with respect to any of the following sets of algebras: $\{\R_{-1}\}$, $\{\Q_{-1}\}$, and $\{\A_{a} \mid a < 0\}$ for any non-trivial Abelian $\ell$-group $\A$. 
\end{thm}

\begin{proof}
We prove the final part of the claim, the rest then follows from Lemma~\ref{l:isomorRQ}. Using Theorem~\ref{t:pAb lots R} and Corollary~\ref{c:ultra-decomposition of semilinear logics}, we know that 
$$
\lalg{\Lu} \subseteq \lalg{\pAb} \subseteq \I(\oa) \cup \ISPU(\{\A_{a} \mid a < 0\}) \cup \ISPU(\{\A_{a} \mid a \geq 0\}).
$$
Clearly, for any non-trivial element of $\ISPU(\{\A_{a} \mid a \geq 0\})$, we have $\fa^{\A_{a}}\geq 0$ (an equation preserved by ultrapowers) and thus it is not a member of $\lalg{\Lu}$ due to the previous lemma. Thus 
$$
\lalg{\Lu} \subseteq \I(\oa) \cup \ISPU(\{\A_{a} \mid a < 0\})
$$
and so the proof is done by Corollary~\ref{c:ultra-decomposition of semilinear logics}.\qedhere 
\end{proof}

It is easy to see that the logic $\Lu$ is not strongly complete with respect to $\R_{-1}$. For instance, the rule $\rul{Arch}^\lor$, introduced in the previous section, is valid in $\R_{-1}$ but is not derivable in \Lu\ (by Example \ref{e:ArchNonArch}). This shows that the full set of validities in $\R_{-1}$ cannot be captured by a finitary system extending $\Lu$ and will therefore require an infinitary axiomatization. To this end, we define now the following extension of $\Lu$:

\begin{defn}
By \emph{infinitary Łukasiewicz unbound logic} $\Lu_{\infty}$ we understand the extension of the logic $\Lu$ by the rule $\rul{Arch}^\lor$.
\end{defn}

Equivalently, $\Lu_{\infty}$ can be seen as the expansion of $\rAb$ to the language of $\pAb$ by the rule $\rul{Lu}$.  

\begin{thm}\label{t:LU completeness}
The logic $\Lu_{\infty}$ is strongly complete with respect to $\R_{-1}$.
\end{thm}

\begin{proof}
By Corollary~\ref{c:semilinearity_criterium}, the logic $\Lu_{\infty}$ is semilinear.
Since $\Lu_{\infty}$ is an expansion of $\rAb$, every linearly ordered $\Lu_{\infty}$-algebra is a pointed Archimedean $\ell$-group. Since $\Lu_{\infty}$ is an extension of $\Lu$, by Lemma~\ref{l:similar extensions}, we know that in any non-trivial $\A$ linearly ordered $\Lu_{\infty}$-algebra we have $\fa^\A < 0$. Thus, $\A$ has to be a pointed Archimedean $\ell$-group with a negative point $a$. Thus, by Hölder's Theorem~\cite{Holder:Axiome}, there is an embedding $h$ of the $\fa$-free reduct of $\A$ into $\R$. Clearly $h$ can be seen as an embedding of $\A$ into $\R_{h(a)}$. As 
$h(a) < 0$, we know that  $\R_{h(a)}$ is isomorphic to $\R_{-1}$ (Lemma~\ref{l:isomorRQ}). Therefore, 
$\lalg{\Lu_{\infty}} \subseteq \IS(\R_{-1})$. Since $\Lu_{\infty}$ is semilinear and $\R_{-1}$ is an $\Lu_{\infty}$-algebra it follows that $\Lu_{\infty}$ is strongly complete with respect to $\R_{-1}$. 
\end{proof}

To conclude our analysis of these systems, we discuss the connection between (infinitary) {\L}ukasiewicz unbound logic and (infinitary) {\L}ukasiewicz logic. The {\L}ukasiewicz logic $\mathrmL$ is a finitary logic axiomatized by four axioms and \emph{modus ponens} and is finitely strongly complete w.r.t.\ the standard \MV-algebra $\mathbfL$ introduced at the beginning of this section. The infinitary {\L}ukasiewicz logic $\mathrmL_\infty$ is its extension by the rule $\rul{Hay}$ introduced in Section~\ref{s:InfAbel}, this logic indeed is not finitary but is strongly complete w.r.t.\ the standard \MV-algebra $\mathbfL$.

As we have seen in the beginning of the section, the algebra $\mathbfL$ is a restriction of the algebra $\alg{LU}$ which in turn is an isomorphic copy of $\R_{-1}$. We use this fact (together with the completeness results for the involved logics), to prove the following results linking (infinitary) {\L}ukasiewicz logic and (infinitary) {\L}ukasiewicz unbound logics via the mapping $\tau$\footnote{A related mapping is presented in \cite[Definition 23]{Metcalfe-Olivetti-Gabbay:SequentCalculiLukasiewicz} as a translation of \L ukasiewicz logic into Abelian logic. The only difference between that translation and the present one is that the former employs an arbitrary fixed propositional variable $q^\perp$ in precisely those places where our construction uses the constant symbol \fa. In fact, more can be said about the connection between $\mathrmL$ and $\Lu$: In \cite[Theorem 20]{Metcalfe-Olivetti-Gabbay:SequentCalculiLukasiewicz}, it is shown that $\mathrmL$ is the $\{\supset,\fa\}$-fragment of $\Lu$, where $\varphi \supset \psi$ is defined as $(\varphi \land \tr) \imp (\psi \lor \fa)$.}
defined recursively on the set of formulas:
\begin{itemize}
\item $\tau(p) = (p \lor \fa) \land \tr$ for all variables $p$.
\item $\tau(\p \imp \x) = (\tau(\p)  \imp \tau(\x)) \land \tr$.
\item $\tau(\p \conj \x) = (\tau(\p)  \conj \tau(\x)) \lor \fa$.
\item $\tau(c) = c$ for $c \in \{\fa,\tr\}$.
\item $\tau(\p \circ \x) = \tau(\p)  \circ \tau(\x)$ for $\circ \in \{\land,\lor\}.$
\end{itemize}

\begin{thm}\label{t:LukLUtranslation}
Let $\Gamma \cup \{\f \}$ be a set of formulas. Then, we have
$$
\Gamma \vdash_{\mathrmL_\infty} \f \qquad\text{ iff }\qquad \tau[\Gamma]\vdash_{\Lu_\infty} \tau(\f) .
$$
Moreover, if $\Gamma$ is finite we have
$$
\Gamma\vdash_{\mathrmL} \f \qquad\text{ iff }\qquad \tau[\Gamma]\vdash_{\Lu} \tau(\f) .$$
\end{thm}

\begin{proof}
Recall that, thanks to the known completeness properties, we can replace the syntactical consequence relations by semantical ones w.r.t.\ algebras $\mathbfL$ and $\alg{\Lu}$ (because it is an isomorphic copy of $\R_{-1}$).

Let us consider any mapping $e$ of propositional variables into reals and
define the mapping $\bar e$ as its restriction to $[0,1]$, i.e., 
$$
\bar e(p) = 
\begin{cases}
     e(p) & \text{ if } e(p) \in [0,1] \\
     1    & \text{ if } e(p) \geq 1 \\
     0    & \text{ if } e(p) \leq 0.
\end{cases}
$$
We denote by $(\bar e)^\mathbfL$ and $e^\alg{Lu}$ the corresponding evaluations. By a simple proof by induction over the complexity of formula $\x$, we obtain:
$$
(\bar e)^\mathbfL(\x) = e^\alg{Lu}(\tau(\x)).
$$
Clearly, it holds for variables and constants by definition of $\bar e$ and $\tau$. 
Now assume $(\bar e)^\mathbfL(\alpha) = e^\alg{Lu}(\tau(\alpha))$ and $(\bar e)^\mathbfL(\beta) = e^\alg{Lu}(\tau(\beta))$ for formulas $\alpha,\beta$. We want to show that $(\bar e)^\mathbfL(\alpha \circ \beta) = e^\alg{Lu}(\tau(\alpha \circ \beta))$ for $\circ \in \{\land,\lor,\conj,\imp\}$. We show this step only for $\circ=\conj$.
We write:
\begin{multline*}  
    e^\alg{Lu}(\tau(\alpha \conj \beta))=e^\alg{Lu}((\tau(\alpha) \conj \tau(\beta)) \lor \fa)=  (e^\alg{Lu}(\tau(\alpha))+e^\alg{Lu}(\tau(\beta))-1) \lor 0= \\ ((\bar e)^\mathbfL(\alpha)+(\bar e)^\mathbfL(\beta)-1) \lor 0=(\bar e)^\mathbfL (\alpha \conj \beta).
\end{multline*}

Now we can prove the left-to-right direction of the first equivalence. Assume that $\Gamma\vDash_{\mathbfL} \f$ and that $e$ is a mapping such that for each $\gamma\in \Gamma$ we have (recall that $\tr^\alg{LU} = 1$)
$$
e^\alg{Lu}(\tau(\gamma)) \geq 1.
$$
As $e^\alg{Lu}(\tau(\gamma)) = (\bar e)^\mathbfL(\gamma)$ we have
$(\bar e)^\mathbfL(\gamma) = 1$ for $\gamma\in\Gamma$ and so, by the assumption, also
$(\bar e)^\mathbfL(\f) = 1$, which entails $e^\alg{\Lu}(\tau(\f)) = 1$.

To prove the converse direction, it suffices to note that if the co-domain of $e$ is a subset of $[0,1]$, then $e = \bar e$. 

The second equivalence can either be proved in exactly the same way or actually follows from the first one and the fact that on finite sets of premises the \L ukasiewicz (unbound) logic coincides with its infinitary variant. 
\end{proof}

The translation provides some information regarding computational complexity. In \L ukasiewicz logic, the problem of deciding the validity of finitary consecutions is known to be coNP-complete \cite{Hanikova:Handbook}. As Theorem~\ref{t:LukLUtranslation} provides a polynomial-time reduction of finitary consecutions from $\mathrmL$ to $\Lu$, it immediately follows that the corresponding problem for $\Lu$ is coNP-hard. The more involved question of whether the set of valid finitary consecutions of $\Lu$ is also in coNP, and thus coNP-complete, has been answered affirmatively in the paper \cite{Hanikova-Jankovec:ComplexityUnboundedRelative}. 

Having established the axiomatization for the logic of $\R_{-1}$, we now turn to the complementary task of characterizing the logics of the other prominent pointed algebras, $\R_0$ and $\R_1$. As we will show, a similar approach allows us to axiomatize these systems in a straightforward manner.
Let us define a new connective $-\f$  as $\f \imp \tr$. By simple checking we prove the following lemma (note that we could prove a more complex variant akin to Lemma~\ref{l:similar extensions} to demonstrate that seemingly simpler rules would not do).

\begin{lemma}
Let $\A \in \lalg{\pAb}$ and $\A \neq \oa$. Then:
\begin{itemize}
\item $\fa^\A < 0$ iff the rule $\fa \lor \f \wdash \f$ is valid in $\A$.
\item $\fa^\A \leq 0$ iff the axiom $\wdash -\fa$ is valid in $\A$.
\item $\fa^\A = 0$ iff the axiom $\wdash \fa\wedge -\fa$ is valid in $\A$.
\item $\fa^\A \neq 0$ iff the rule $(\fa\wedge -\fa) \lor \f \wdash \f$ is valid in $\A$.
\item $\fa^\A \geq 0$ iff the axiom $\wdash \fa$ is valid in $\A$.
\item $\fa^\A > 0$ iff the rule $-\fa \lor \f \wdash \f$ is valid in $\A$.
\end{itemize}
\end{lemma}

\begin{thm}
Let $\alg{A}$ be any non-trivial $\rAb$-algebra. Then the extension of $\pAb$ by the rule/axiom mentioned in a given row of Table~\ref{tab:compl} is finitely strongly complete w.r.t.\ the sets of algebras mentioned in the corresponding columns Set 1--3. Its extension by the rule $\rul{Arch}^\vee$ is strongly complete w.r.t.\ the set of algebras mentioned in the corresponding column Set 3.
\end{thm}

\begin{table}
\begin{center}
\begin{tabular}{c|c@{\quad \ }c@{\quad \ }c}
Extension of $\pAb$ by & Set 1 & Set 2 & Set 3 \\\hline
$\emptyset$ & $\{\A_a \mid a\in A\}$ & $\{\Q_{-1}, \Q_0, \Q_1\}$ & $\{\R_{-1}, \R_0, \R_1\}$\\
$\wdash \fa$ & $\{\A_a \mid a \geq 0\}$ & $\{\Q_{0}, \Q_1\}$ & $\{\R_{0}, \R_1\}$\\
$\wdash -\fa$ & $\{\A_a \mid a \leq 0\}$ & $\{\Q_{-1}, \Q_0\}$ & $\{\R_{-1}, \R_0\}$\\
$(\fa\wedge -\fa)\lor \f \wdash \f$ & $\{\A_a \mid a \neq 0\}$ & $\{\Q_{-1}, \Q_{1}\}$ & $\{\R_{-1}, \R_{1}\}$\\
$\fa \lor \f \wdash \f$ & $\{\A_a \mid a < 0\}$ & $\{\Q_{-1}\}$ & $\{\R_{-1}\}$\\
$-\fa \lor \f \wdash \f$ & $\{\A_a \mid a > 0\}$ & $\{\Q_{1}\}$ & $\{\R_{1}\}$\\
$\wdash \fa\wedge-\fa$ & $\{\A_0\} $ & $\{\Q_0\}$ & $\{\R_0\}$
\end{tabular}
\end{center}
\caption{Completeness properties of prominent extensions of $\pAb$}\label{tab:compl}
\end{table}

As mentioned before, the expansion of $\pAb$ by $\wdash \fa\wedge-\fa$ can be seen as the Abelian logic itself (formally speaking, these two logics are termwise equivalent by setting $\fa = \tr$).

Let us conclude this section by showing a natural and simple translation between $\Lu$ and the expansion of $\pAb$ by $-\fa \lor \f \wdash \f$, denoted here by $\Lu^*$.

\begin{thm}\label{t:tau}
Let us consider the mapping $\tau\colon\Fm{\lang{L_\pAb}} \longrightarrow \Fm{\lang{L_\pAb}}$ defined recursively as follows:
\begin{itemize}
\item $\tau\colon \fa \mapsto -\fa$.
\item $\tau\colon \tr \mapsto \tr$.
\item $\tau\colon x \mapsto x $ for each variable $x$.
\item $\tau\colon (\p \circ \xi) \mapsto (\tau(\p) \circ \tau(\xi))$ for all binary connectives $\circ \in \lang{L_\pAb}$.
\end{itemize}
For any consecution $\Gamma \wdash \f$ and for each Abelian $\ell$-group $\A$, and $a \in A$ we have
    
$$\Gamma\vDash_{\A_a} \f \qquad \text{iff} \qquad \tau[\Gamma]\vDash_{\A_{-a}} \tau(\f).$$

In particular, for any consecution $\Gamma \wdash \f$ we have
\begin{enumerate}    
\item $\Gamma\vdash_{\Lu} \f$ iff $\tau[\Gamma]\vdash_{\Lu^*} \tau(\f) $.
\item $\Gamma\vdash_{\Lu^*} \f$ iff $\tau[\Gamma]\vdash_{\Lu} \tau(\f) $.
\end{enumerate}
\end{thm}

\begin{proof}

Consider an arbitrary $\A_a$-evaluation $e_1\colon\Fm{L_{\pAb}} \longrightarrow \A_a$. Clearly, $e_1$ extends the evaluation $e_0\colon\Fm{L_{\Ab}} \longrightarrow \A$. Clearly there is a unique evaluation $e_2\colon\Fm{L_{\pAb}} \longrightarrow \A_{-a}$ extending $e_0$ such that $e_1(\fa)=-e_2(\fa)=e_2(-\fa)$ and $e_1(x)=e_2(x)$ for all variables $x$. By induction one can show that $e_1(\p) =e_2(\tau(\p) )$ for each formula $\p \in \Fm{L_{\pAb}}$.
Therefore, $e_1(\p) \geq 0$ iff $e_2(\tau(\p) ) \geq 0$. From this, one can easily derive the statement of this theorem.
\end{proof}

Clearly, the same translation would link the infinitary versions of these two logics (i.e., their extensions of $\rul{Arch}^\vee$); the expansions of $\pAb$ by (1) by axiom $\fa$ and (2) axiom $-\fa$; and their extensions of $\rul{Arch}^\vee$.

\section{Future Work}\label{s:Conclusions}

In this section, we outline possible directions for future research on the topics covered in this paper.

\begin{itemize}
\item  Study of the lattice of infinitary extensions of Abelian logic: While the present paper provides an infinitary rule for the extension corresponding to the generalized quasivariety generated by $\R$ and establishes the existence of $2^{2^\omega}$ distinct infinitary extensions, it is clear that we have only started the topic and the structure of the full lattice of infinitary extensions of \Ab\ remains largely unexplored.
One of the goals for future research is a systematic classification of this lattice. In particular, a major open challenge is explicitly axiomatizing other fundamental logics, namely the logics of \Z\ and \Q. 

\item Infinitary logics of distinct total orders on $\Z^n$: It is a well-established result \cite{Clay-Rolfsen:OrderedGroups} that the space of total orders on $\Z^n$ is homeomorphic to a Cantor space for $n>1$. A natural open problem is to determine whether each distinct ordering within this space generates a unique infinitary logic.

 \item Study of finitary extensions of pointed Abelian logic: In the case of \pAb, one should first focus on {\em finitary} extensions, as they are plentiful. In this paper we have axiomatized several prominent logics of pointed Abelian $\ell$-groups, but a systematic study remains an interesting topic for the future. On this matter, it is relevant to mention that the relation between \L ukasiewicz logic and the \L ukasiewicz unbound logic proved in Theorem~\ref{t:LukLUtranslation} also holds between the $n$-valued \L ukasiewicz logic and the logic of $\Z_{n-1}$ (and analogously between other axiomatic extensions of \L ukasiewicz logic and properly chosen extensions of $\Lu$). Algebraically speaking, the lattice of subvarieties of \L ukasiewicz logic (described by Komori in~\cite{Komori:SuperLukasiewiczPropositional}) can be embedded into the lattice of subvarieties of $\Lu$ (first shown in \cite{Young:Varieties_of_pointed_Abelian_l-groups}). The submitted paper \cite{Jankovec:SubvarietiesPointedAbelianLGroups} describes all quasivarieties generated by linearly ordered pointed Abelian $\ell$-groups. Nevertheless, the case of non-semilinear finitary extensions of pointed Abelian logic remains unexplored.\looseness-1

\item Semilinearity and completeness: Another area for further investigation is to determine whether certain extensions of Abelian logic, specifically the logic $\Ab+\rul{IDC}^\vee+\rul{Arch}^\vee$, are semilinear, which would imply strong completeness with respect to $\Z$. We have conjectured that this logic may be strongly complete with respect to $\Z$, but this remains to be proven.

\item Study of other superabelian logics: In this paper we have focused on \L ukasiewicz unbound logic and infinitary extensions of Abelian logic, but the framework from Section \ref{s:Prelim} applies to any superabelian logic, e.g.\ certain modal logics or the expansions of $\Ab$ with connectives corresponding to additional arithmetical operations on real numbers, such as multiplication or division. 
\end{itemize}

\vspace{-1ex}
\section*{Acknowledgements}
The authors thank the editor and the referees for their careful reading and comments that contributed substantially to the improvement of the article.

\vspace{-1ex}
\section*{Funding}
All authors were partly supported by European Union's Marie Sklodowska--Curie grant no.\ 101007627 (MOSAIC project). The first two authors (Cintula and Jankovec) were also supported by a grant from the Programme Johannes Amos Comenius under the Ministry of Education, Youth and Sports of the Czech Republic, CZ.02.01.01/00/23\_025/0008711.
Finally, the second author (Jankovec) was also supported by the project SVV-2025-260837.

\bibliographystyle{plainurl}
\bibliography{mfl}

\begin{thebibliography}{10}

\bibitem{Bergman:UniversalAlgebra}
Clifford Bergman.
\newblock {\em Universal algebra}, volume 301 of {\em Pure and Applied
  Mathematics (Boca Raton)}.
\newblock CRC Press, Boca Raton, FL, 2012.
\newblock Fundamentals and selected topics.

\bibitem{Birkhoff:LatticeTheory}
Garrett Birkhoff.
\newblock {\em Lattice Theory}, volume~25 of {\em American Mathematical Society
  Colloquium Publications}.
\newblock American Mathematical Society, Providence, third edition, 1967.

\bibitem{Blok-Pigozzi:AlgebraizableLogics}
Willem~J. Blok and Don~L. Pigozzi.
\newblock {\em Algebraizable Logics}, volume 396 of {\em Memoirs of the
  American Mathematical Society}.
\newblock American Mathematical Society, Providence, 1989.

\bibitem{Burris-Sankappanavar:CourseUniversalAlgebra}
Stanley Burris and H.P. Sankappanavar.
\newblock {\em A Course in Universal Algebra}, volume~78 of {\em Graduate Texts
  in Mathematics}.
\newblock Springer, New York, 1981.

\bibitem{Casari:ComparativeLogics}
Ettore Casari.
\newblock Comparative logics and {A}belian {$\ell$}-groups.
\newblock In R.~Ferro, C.~Bonotto, S.~Valentini, and A.~Zanardo, editors, {\em
  Logic {C}olloquium '88}, volume 127 of {\em Studies in Logic and the
  Foundations of Mathematics}, pages 161--190. North-Holland, Amsterdam, 1989.

\bibitem{Cignoli-Ottaviano-Mundici:AlgebraicFoundations}
Roberto Cignoli, Itala~M.L. D'Ottaviano, and Daniele Mundici.
\newblock {\em Algebraic Foundations of Many-Valued Reasoning}, volume~7 of
  {\em Trends in Logic}.
\newblock Kluwer, Dordrecht, 1999.

\bibitem{Cintula:WIFL-BasicProperties}
Petr Cintula.
\newblock Weakly implicative (fuzzy) logics~{I}: {B}asic properties.
\newblock {\em Archive for Mathematical Logic}, 45(6):673--704, 2006.

\bibitem{Cintula-Grimau-Noguera-Smith:DegreesTo11}
Petr Cintula, Berta Grimau, Carles Noguera, and Nicholas J.~J. Smith.
\newblock These degrees go to eleven: fuzzy logics and gradable predicates.
\newblock {\em Synthese}, 200(6):Paper No. 445, 38, 2022.
\newblock \href {https://doi.org/10.1007/s11229-022-03909-2}
  {\path{doi:10.1007/s11229-022-03909-2}}.

\bibitem{Cintula-Noguera:TheBook}
Petr Cintula and Carles Noguera.
\newblock {\em Logic and {I}mplication: {A}n {I}ntroduction to the {G}eneral
  {A}lgebraic {S}tudy of {N}on-classical {L}ogics}, volume~57 of {\em Trends in
  Logic}.
\newblock Springer, 2021.

\bibitem{Clay-Rolfsen:OrderedGroups}
A.~Clay and D.~Rolfsen.
\newblock {\em Ordered Groups and Topology}.
\newblock Graduate Studies in Mathematics. American Mathematical Society, 2016.

\bibitem{Font:AALBook}
Josep~Maria Font.
\newblock {\em Abstract Algebraic Logic. {A}n Introductory Textbook}, volume~60
  of {\em Studies in Logic, Mathematical Logic and Foundations}.
\newblock College Publications, London, 2016.

\bibitem{Font-Jansana-Pigozzi:SurveyAAL}
Josep~Maria Font, Ramon Jansana, and Don~L. Pigozzi.
\newblock A survey of abstract algebraic logic.
\newblock {\em Studia Logica}, 74(1--2):13--97, 2003.

\bibitem{Fuchs:PartiallyOrdered}
L{\'{a}}szl{\'{o}} Fuchs.
\newblock {\em Partially Ordered Algebraic Systems}.
\newblock Pergamon Press, Oxford, 1963.

\bibitem{Gabbay-Metcalfe:ContinuousUninorms}
Dov~M. Gabbay and George Metcalfe.
\newblock Fuzzy logics based on $[0,1)$-continuous uninorms.
\newblock {\em Archive for Mathematical Logic}, 46(6):425--469, 2007.

\bibitem{Galatos-JKO:ResiduatedLattices}
Nikolaos Galatos, Peter Jipsen, Tomasz Kowalski, and Hiroakira Ono.
\newblock {\em Residuated Lattices: {A}n Algebraic Glimpse at Substructural
  Logics}, volume 151 of {\em Studies in Logic and the Foundations of
  Mathematics}.
\newblock Elsevier, Amsterdam, 2007.

\bibitem{Gobrunov:Quasivarieties}
Viktor~A. Gorbunov.
\newblock {\em Algebraic Theory of Quasivarieties}.
\newblock Siberian School of Algebra and Logic. Consultants Bureau, New York,
  1998.

\bibitem{Hanikova:Handbook}
Zuzana Hanikov{\'{a}}.
\newblock Computational complexity of propositional fuzzy logics.
\newblock In Petr Cintula, Petr H{\'{a}}jek, and Carles Noguera, editors, {\em
  Handbook of Mathematical Fuzzy Logic - Volume 2}, volume~38 of {\em Studies
  in Logic, Mathematical Logic and Foundations}, pages 793--851. College
  Publications, London, 2011.

\bibitem{Hanikova-Jankovec:ComplexityUnboundedRelative}
Zuzana Hanikov\'{a} and Filip Jankovec.
\newblock {Satisfiability in {\L}ukasiewicz Logic and Its Unbounded Relative}.
\newblock In S.~Guerrini and B.~K\"{o}nig, editors, {\em 34th EACSL Annual
  Conference on Computer Science Logic}, volume 363 of {\em LIPIcs}, pages
  14:1--14:18. Schloss Dagstuhl -- Leibniz-Zentrum f{\"u}r Informatik, 2026.
\newblock \href {https://doi.org/10.4230/LIPIcs.CSL.2026.14}
  {\path{doi:10.4230/LIPIcs.CSL.2026.14}}.

\bibitem{Hay:Axiomatization}
Louise~Schmir Hay.
\newblock Axiomatization of the infinite-valued predicate calculus.
\newblock {\em Journal of Symbolic Logic}, 28(1):77--86, 1963.

\bibitem{Holder:Axiome}
Otto H{\"{o}}lder.
\newblock {Die {A}xiome der {Q}uantit{\"{a}}t und die {L}ehre vom {M}ass}.
\newblock {\em Berichte über die Verhandlungen der Königlich Sächsischen
  Gesellschaft der Wissenschaften zu Leipzig, Mathematisch-Physische Klasse},
  53:1--64, 1901.

\bibitem{Jankovec:SubvarietiesPointedAbelianLGroups}
Filip Jankovec.
\newblock {Subvarieties of Pointed Abelian $\ell$-groups}.
\newblock submitted, 2025.
\newblock \href {https://arxiv.org/abs/2509.05044} {\path{arXiv:2509.05044}}.

\bibitem{Khisamiev:UniversalTheoryAbelianGroups}
N.~G. Khisamiev.
\newblock Universalnaya teoriya strukturno uporyadochennykh abelevykh grupp
  ({U}niversal theory of lattice-ordered {A}belian groups).
\newblock {\em Algebra i Logika}, 5(3):71--76, 1966.

\bibitem{Komori:SuperLukasiewiczPropositional}
Yuichi Komori.
\newblock Super-{Ł}ukasiewicz propositional logics.
\newblock {\em Nagoya Mathematical Journal}, 84:119--133, 1981.

\bibitem{DiNola-Leustean:Handbook}
Ioana Leu{\c{s}}tean and Antonio {Di Nola}.
\newblock {Ł}ukasiewicz logic and {MV}-algebras.
\newblock In Petr Cintula, Petr H{\'{a}}jek, and Carles Noguera, editors, {\em
  Handbook of Mathematical Fuzzy Logic - Volume 2}, volume~38 of {\em Studies
  in Logic, Mathematical Logic and Foundations}, pages 469--583. College
  Publications, London, 2011.

\bibitem{Maltsev:Metamathematics}
Anatoli{\v{i}}~Ivanovi{\v{c}} Mal'\hspace{-1pt}tcev.
\newblock {\em The Metamathematics of Algebraic Systems, {C}ollected Papers:
  1936--1967}, volume~66 of {\em Studies in Logic and the Foundations of
  Mathematics}.
\newblock North-Holland, Amsterdam, 1971.

\bibitem{Metcalfe-Olivetti-Gabbay:SequentCalculiLukasiewicz}
George Metcalfe, Nicola Olivetti, and Dov Gabbay.
\newblock Sequent and hypersequent calculi for {A}belian and {Ł}ukasiewicz
  logics.
\newblock {\em {ACM} Transactions on Computational Logic}, 6(3):578--613, 2005.

\bibitem{Meyer-Slaney:AbelianLogic}
Robert~K. Meyer and John~K. Slaney.
\newblock Abelian logic from {A} to {Z}.
\newblock In Graham Priest, Richard Routley, and Jean Norman, editors, {\em
  Paraconsistent Logic: {E}ssays on the Inconsistent}, Philosophia Analytica,
  pages 245--288. Philosophia Verlag, Munich, 1989.

\bibitem{Mundici-LogicUlamGame}
Daniele Mundici.
\newblock The logic of {U}lam's game with lies.
\newblock In Cristina Bicchieri and M.L. {Dalla Chiara}, editors, {\em
  Knowledge, Belief, and Strategic Interaction ({C}astiglioncello, 1989)},
  Cambridge Studies in Probability, Induction, and Decision Theory, pages
  275--284. Cambridge University Press, Cambridge, 1992.

\bibitem{Young:Varieties_of_pointed_Abelian_l-groups}
William Young.
\newblock Varieties generated by unital {A}belian {$\ell$}-groups.
\newblock {\em Journal of Pure and Applied Algebra}, 219(1):161--169, 2015.
\newblock \href {https://doi.org/10.1016/j.jpaa.2014.04.015}
  {\path{doi:10.1016/j.jpaa.2014.04.015}}.

\bibitem{Lukasiewicz-Tarski:Untersuchungen}
Jan {Ł}ukasiewicz and Alfred Tarski.
\newblock Untersuchungen {\"{u}}ber den {A}ussagenkalk{\"{u}}l.
\newblock {\em Comptes Rendus des S{\'{e}}ances de la {S}oci{\'{e}}t{\'{e}} des
  Sciences et des {L}ettres de {V}arsovie, {C}lasse {III}}, 23:30--50, 1930.

\end{thebibliography}

\end{document}